\documentclass[11pt,reqno]{amsart}

\usepackage[utf8]{inputenc}
\usepackage[english]{babel}
\usepackage{ifthen} 

\usepackage{amsmath} 
\usepackage{amssymb} 
\usepackage{mathrsfs} 
\usepackage{mathtools} 
\usepackage{bbm} 
\usepackage{cancel} 
\usepackage{mathdots} 
\usepackage{xfrac} 

\usepackage{amsthm} 
\usepackage[shortlabels]{enumitem}

\usepackage{lipsum} 
\usepackage{verbatim} 
\usepackage{pifont} 
\usepackage[dvipsnames]{xcolor} 
\usepackage{lmodern} 
\usepackage{marginnote} 

\usepackage{wrapfig} 
\usepackage{caption} 

\usepackage{array} 
\usepackage{multirow} 
\usepackage{multicol} 
\usepackage{colortbl} 
\usepackage{makecell} 
\usepackage{arydshln} 

\usepackage{tikz} 
\usetikzlibrary{automata,positioning} 
\usepackage{tkz-euclide} 
\usepackage[customcolors,shade]{hf-tikz} 

\usepackage{geometry} 
\usepackage{titleps} 
\usepackage{footnote} 

\usepackage{hyperref} 
\usepackage{cleveref} 

\newtheorem{theorem}{Theorem}[section]
\newtheorem{proposition}[theorem]{Proposition}
\newtheorem{lemma}[theorem]{Lemma}
\newtheorem{corollary}[theorem]{Corollary}

\newtheorem{problem}{Problem}
\newtheorem{remark}[theorem]{Remark}

\theoremstyle{definition}
\newtheorem{definition}[theorem]{Definition}

\newcolumntype{L}[1]{>{\raggedright\let\newline\\\arraybackslash\hspace{0pt}}m{#1}}
\newcolumntype{C}[1]{>{\centering\let\newline\\\arraybackslash\hspace{0pt}}m{#1}}
\newcolumntype{R}[1]{>{\raggedleft\let\newline\\\arraybackslash\hspace{0pt}}m{#1}}

\usepackage{float}

\hypersetup{
    colorlinks=true,
    linkcolor=blue,
    filecolor=magenta,
    urlcolor=red,
    citecolor=blue,
}
\usetikzlibrary{decorations.pathmorphing}
\usetikzlibrary{decorations.markings}
\usetikzlibrary{decorations.text}

\usetikzlibrary{fit}
\usetikzlibrary{backgrounds}
\usepackage{subcaption}
\usepackage{adjustbox} 

\definecolor{cfefefe}{RGB}{254,254,254}
\definecolor{c0b0b0b}{RGB}{11,11,11}
\definecolor{cc7c6c3}{RGB}{199,198,195}
\definecolor{cdfdedf}{RGB}{223,222,223}
\definecolor{c0a0a09}{RGB}{10,10,9}
\definecolor{c0a0a0a}{RGB}{10,10,10}
\definecolor{c080808}{RGB}{8,8,8}
\definecolor{cbfbcb6}{RGB}{191,188,182}
\definecolor{cc1beb7}{RGB}{193,190,183}
\definecolor{cb5b4b1}{RGB}{181,180,177}
\definecolor{cd8d8d8}{RGB}{216,216,216}
\definecolor{cb8b7b3}{RGB}{184,183,179}
\definecolor{cc4c2c0}{RGB}{196,194,192}
\definecolor{c1b1b1a}{RGB}{27,27,26}
\definecolor{c9e9991}{RGB}{158,153,145}
\definecolor{cbdbab4}{RGB}{189,186,180}
\definecolor{cb9b7b4}{RGB}{185,183,180}
\definecolor{cb9b8b5}{RGB}{185,184,181}
\definecolor{c1a1a19}{RGB}{26,26,25}
\definecolor{ccfcecb}{RGB}{207,206,203}
\definecolor{cb3aea7}{RGB}{179,174,167}
\definecolor{cc9c7c1}{RGB}{201,199,193}
\definecolor{c6b6a68}{RGB}{107,106,104}
\definecolor{cd3d2cf}{RGB}{211,210,207}
\definecolor{c8b8b8b}{RGB}{139,139,139}
\definecolor{ca6a5a2}{RGB}{166,165,162}
\definecolor{caaa9a6}{RGB}{170,169,166}
\definecolor{c6e6c67}{RGB}{110,108,103}
\definecolor{ca5a4a1}{RGB}{165,164,161}
\definecolor{c9d9c9a}{RGB}{157,156,154}
\definecolor{cb8b7b4}{RGB}{184,183,180}
\definecolor{c1c1c1b}{RGB}{28,28,27}
\definecolor{ca39e96}{RGB}{163,158,150}
\definecolor{ca8a6a0}{RGB}{168,166,160}
\definecolor{cbcbbb8}{RGB}{188,187,184}
\definecolor{c505050}{RGB}{80,80,80}
\definecolor{c9f9b93}{RGB}{159,155,147}
\definecolor{cb4b3b0}{RGB}{180,179,176}
\definecolor{c2c2c2b}{RGB}{44,44,43}
\definecolor{ce2e2e3}{RGB}{226,226,227}
\definecolor{ccbc9c3}{RGB}{203,201,195}

\title[Threshold Graphs are Globally Synchronizing]{Global synchronization beyond dense graphs: the case of threshold graphs}
\author[Wu]{Hongjin Wu}
\address{Centrum Wiskunde \& Informatica, Netherlands}
\email{hongjin-wu@outlook.com}
\author[Brandes]{Ulrik Brandes}
\address{ETH Z\"urich, Switzerland}
\email{ubrandes@ethz.ch}
\date{\today}
\begin{document}

\begin{abstract}
Given a graph \(G\) with adjacency matrix \(A\), consider the homogeneous Kuramoto energy $E_G(\boldsymbol{\theta})
:=
\frac{1}{2}
\sum_{1\leq i,j\leq n}
A_{ij}\bigl(1-\cos(\theta_i-\theta_j)\bigr)$.
We call \(G\) \emph{second-order globally synchronizing} if every second-order stationary point of \(E_G\) is fully synchronized. This property implies \emph{global synchronization}, namely that, up to a measure-zero set of initial conditions, trajectories of the Kuramoto model converge to a fully synchronized state.
A fundamental graph-theoretic question is to identify which graph structures have this property. Existing guarantees for global synchronization typically require large minimum degree which forces the graph to be very dense, or good expansion properties. In this paper, we show that synchronization can also arise from a different, purely structural mechanism. More precisely, we prove that threshold graphs, a classical recursively defined graph class, are second-order globally synchronizing, and hence globally synchronizing. Thus, globally synchronizing graphs need not be very dense, have large minimum degree, or satisfy strong expansion-type conditions. The proof exploits the recursive construction of threshold graphs: local phasor constraints imposed by second-order stationarity are propagated along the construction sequence until full synchronization is forced.
\end{abstract}

\maketitle
\tableofcontents
\section{Introduction}\label{sec:intro}
Understanding how graph structure influences synchronization phenomena is a central theme at the interface of graph theory, network dynamics, and nonlinear systems. Synchronization itself is a classical phenomenon: its study dates back to the 17th century, when Christiaan Huygens observed the spontaneous synchronization of pendulum clocks. Since then, synchronization phenomena have been widely studied in natural and technological systems, including the flashing of fireflies \cite{Buck1968, Strogatz1993}, synchronization in power grids \cite{dorfler2012synchronization}, and more recent connections with machine-learning models \cite{Miyato2024ArtificialKuramoto,Geshkovski2025,Criscitiello2024}.

Among mathematical models of synchronization, the Kuramoto model has become one of the standard frameworks because it combines analytical tractability with rich nonlinear behavior. Introduced by Yoshiki Kuramoto in 1975 \cite{Kuramoto1975}, the model describes \(n\) oscillators on the unit circle, coupled according to a graph \(G=(V,E)\) with adjacency matrix \(\boldsymbol{A}\). It is given by the system of ordinary differential equations
\[
 \frac{d\theta_i}{dt}
 =
 \omega_i
 +
 \sum_{j=1}^{n} \boldsymbol{A}_{ij}\sin(\theta_j-\theta_i),
 \qquad i\in V,
\]
where \(\theta_i:\mathbb{R}\to\mathbb{S}^1\) denotes the phase of oscillator \(i\), and \(\omega_i\) is its intrinsic frequency. Thus, each oscillator evolves according to its own intrinsic frequency together with nonlinear interactions determined by phase differences with its neighbors.

In the general heterogeneous setting, where the intrinsic frequencies \(\omega_i\) may differ, a central question is whether the oscillators achieve frequency synchronization. In contrast, the discrete structure of the coupling graph becomes especially prominent in the homogeneous setting, where all intrinsic frequencies are identical, i.e.,
\(\omega_i=\omega\) for all \(i\in\{1,\ldots,n\}\).
After passing to a co-rotating frame, one may assume \(\omega_i=0\) for all \(i\), and the dynamics reduce to
\begin{equation}\label{kuramoto}
 \frac{d\theta_i}{dt}
 =
 \sum_{j=1}^{n} \boldsymbol{A}_{ij}\sin(\theta_j-\theta_i),
 \qquad i\in V.
\end{equation}
In this setting, frequency synchronization is automatic, and the remaining question is phase synchronization: whether all phases converge to a common value. Thus, the problem becomes genuinely graph-theoretic: Which graph structures force synchronization from almost all initial conditions?
\medskip

Now we begin with a precise definition of globally synchronizing graphs.

\begin{definition}[Globally synchronizing graph]\label{defi-globalsync}
We say that a graph \(G=(V,E)\) is \emph{globally synchronizing} if,
for all initial conditions \(\boldsymbol{\theta}(0)\in\mathbb{R}^n\),
except for a set of measure zero,
the solution to the Kuramoto model~\eqref{kuramoto} on $G$ satisfies
\[
\lim_{t\to\infty} \bigl(\theta_i(t)-\theta_j(t)\bigr)=0
\pmod{2\pi}
\qquad \text{for all } i,j\in V.
\]
\end{definition}

To build some intuition for the model~\eqref{kuramoto} and the
definition~\eqref{defi-globalsync}, one may imagine placing $n$ oscillators on
the unit circle and connecting them by edges specified by the adjacency matrix
$A$. If there is an edge between nodes $i$ and $j$, then whenever their phases
differ, the interaction tends to reduce this difference and pull them toward a
common phase.

\begin{figure}[H]
\centering
\begin{tikzpicture}[scale=2]
\tikzset{
  vertex/.style={circle, draw=black, fill=black, minimum size=8pt, inner sep=0pt},
  elabel/.style={midway, sloped, draw=none, fill=none, text=black, font=\small}
}

\draw[thick] (0,0) circle (1);

\node[vertex, label=above:{$1$}] (n2) at (90:1)  {};
\node[vertex, label=right:{$2$}] (n3) at (330:1) {};
\node[vertex, label=left:{$3$}]  (n6) at (190:1) {};

\draw[decorate, decoration={coil,aspect=0.4,segment length=3pt,amplitude=2pt}]
  (n2) -- node[elabel, above] {$A_{12}$} (n3);

\draw[decorate, decoration={coil,aspect=0.4,segment length=3pt,amplitude=2pt}]
  (n2) -- node[elabel, above] {$A_{13}$} (n6);

\draw[decorate, decoration={coil,aspect=0.4,segment length=3pt,amplitude=2pt}]
  (n3) -- node[elabel, below] {$A_{23}$} (n6);
\draw[->, thick, black] (95:1.05) arc (95:120:1.05) node[midway, above, text=black, font=\small] {$3$};
\draw[->, thick, black] (185:1.05) arc (185:160:1.05) node[midway, left, text=black, font=\small] {$1$};
\draw[->, thick, black] (85:1.05) arc (85:60:1.05) node[midway, above, text=black, font=\small] {$2$} ;
\draw[->, thick, black] (335:1.05) arc (335:360:1.05) node[midway, right, text=black, font=\small] {$1$};
\draw[->, thick, black] (195:1.05) arc (195:220:1.05) node[midway, left, text=black, font=\small] {$2$};
\draw[->, thick, black] (325:1.05) arc (325:300:1.05) node[midway, right, text=black, font=\small] {$3$};
\end{tikzpicture}
\caption{A spring analogy for a coupled oscillator network. The label on each curved arrow marks which node exerts the force.}
\label{oscillators}
\end{figure}
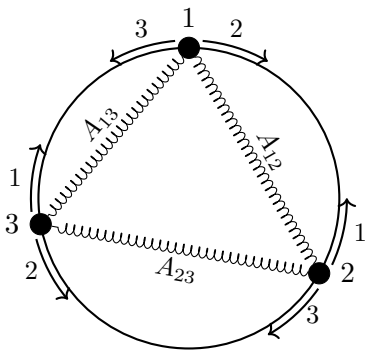
From the above illustration of the complete graph $K_3$, one may ask whether
the three oscillators will eventually converge to a common position when placed
as shown in Figure \ref{oscillators}.
Each node is subject to two forces, for instance, node $1$ is pulled
toward node $2$ on one side and toward node $3$ on the other.
So it is not
immediately clear in which direction it will move, let alone whether all three
nodes will ultimately coincide.
Even in this simple example, the difficulty of the problem becomes apparent.
Nevertheless, one may form an intuitive if
somewhat heuristic expectation: since in a complete graph each oscillator tends
to move closer to every other oscillator, the abundance of interactions may
indeed drive the system to synchronize from almost all initial configurations.

One standard approach to this problem is to study the corresponding energy landscape.
Precisely, the Kuramoto model \eqref{kuramoto} defines a \textit{gradient dynamical system} (or called \textit{gradient flow}\footnote{We call a dynamical system a \emph{gradient flow} on $E(x)$ if it has the form $\frac{dx(t)}{dt} = -\nabla E(x)$. Its trajectories evolve along the steepest descent direction of $E$.}) with energy function \( E_G (\boldsymbol{\theta})\), satisfying
$\frac{d\boldsymbol{\theta}}{dt} = -\nabla E_G(\boldsymbol{\theta})$ where \( E_G(\boldsymbol{\theta}) : \mathbb R^n \rightarrow \mathbb{R} \) with graph \( G=(V,E) \) is defined as
\begin{equation}\label{E}
\begin{aligned}
    E_G(\boldsymbol\theta) := &\frac{1}{2} \sum_{1\leq i,j\leq n} \boldsymbol{A}_{ij} \big(1-\cos(\theta_i - \theta_j)\big). \\
    \end{aligned}
\end{equation}
Observe that the energy function \eqref{E} is nonconvex in general. The fully synchronized states are precisely its global minimizers. However, nonconvexity allows the possibility of nonsynchronized equilibria satisfying second-order necessary conditions for local optimality. Such points constitute a potential obstruction to global synchronization: if they correspond to stable equilibria of the gradient dynamics, then trajectories initialized in their basin of attraction may converge to nonsynchronized states.

A sufficient way to rule out this obstruction is to exclude nonsynchronized second-order stationary points. More precisely, suppose that every state satisfying the first and second-order stationary condition is a fully synchronized state, i.e., $\theta_1=\cdots=\theta_n$. Then \(G\) is globally synchronizing. This can be shown by the Łojasiewicz gradient inequality \cite{Lageman2007} together with the center manifold theorem; see Lemma A.1 in \cite{Geshkovski2025} for an excellent exposition. Therefore, to prove that a graph is globally synchronizing, it is enough to show that every second-order stationary point of \(E_G\) is synchronized.
This motivates the following graph property.

\begin{definition}[Second-order globally synchronizing graph]\label{defi-second-order-gs}
A graph \(G=(V,E)\) is called \emph{second-order globally synchronizing} if every
configuration \(\boldsymbol{\theta}\) satisfying
\[
    \nabla E_G(\boldsymbol{\theta})=0,
    \qquad
    \nabla^2 E_G(\boldsymbol{\theta})\succeq 0
\]
is a fully synchronized state.
\end{definition}
As discussed above, every second-order globally synchronizing graph is globally
synchronizing. In this paper, we therefore focus on the stronger problem of
characterizing second-order globally synchronizing graphs. This leads to the
central graph-theoretic question studied in this paper:
\begin{center}
\textbf{Key question.}
\emph{Which graphs are second-order globally synchronizing?}
\end{center}
Our main result answers this question for threshold graphs.

\subsection{Prior Work and Our Contributions} 
A dominant line of research focuses on the so-called \textit{extremal combinatorics} version of the problem.
It starts from the fact that complete graphs are globally synchronizing and asks whether sufficiently dense subgraphs of the complete graph still synchronize globally. 
\medskip

\begin{problem}
	Let $\delta(G)$ denote the minimum degree of $G$. Find the smallest constant $\mu_c\in[0,1]$ such that as long as $\delta(G)\ \ge\ \mu_c (n-1)$, the graph is globally synchronizing, whereas below this threshold, there exist graphs satisfying the corresponding minimum-degree bound that fail to be globally synchronizing.
\end{problem}
\medskip

Several advances have been made in \cite{taylor2012there, ling2019landscape, lu2021synchronization,kassabov2021sufficiently,canale2022weighted,yoneda2021lower}.
The current best known bounds on the critical minimum-degree threshold $\mu_c$ satisfy $0.6875 \le \mu_c \le 0.75$, where the upper bound is established in \cite{kassabov2021sufficiently}, and the strongest known lower bound follows from \cite{canale2022weighted}. 
It has been conjectured that below the \(75\%\) minimum-degree threshold, even
highly connected graphs may fail to synchronize. More precisely, for every
\(\epsilon>0\), there exists \(n\) and a graph on \(n\) vertices with minimum
degree at least \((\frac{3}{4}-\epsilon)n\) that is not globally
synchronizing; see~\cite{TownsendStillmanStrogatz2020} and Conjecture~5
in~\cite{bandeira2025randomstrasse101openproblems2024}.
Note that sufficiently large threshold for minimum degree enforces high density. Indeed, when $\mu_c \ge 75\%$, the total number of edges satisfies $|E| \ge \frac{3}{8}\, n(n-1)$, which scales as $O(n^2)$.

Must globally synchronizing graphs be very dense, or have minimum degree exceeding a fixed threshold? Several known examples suggest that this is not the case.
For instance, trees are known to be globally synchronizing. An $n$-vertex tree has $n-1$ edges, and hence edge density $\frac{n-1}{\binom{n}{2}}=\frac{2}{n}$ and minimum degree is $1$.
Also, wheel graphs are globally synchronizing \cite{canale2010wheels}, and their edge density is \(4/n\), and their minimum degree is \(3\).
This suggests that, beyond the extremely dense regime, global synchronization is governed more by structural properties of the graph than by edge density alone.
However, there are still relatively few results supporting this perspective.
\begin{center}
\vspace{-0.4em}
\begin{tikzpicture}[>=stealth, xscale=0.92, yscale=0.75]
\draw[->] (0,0) -- (9,0) node[right] {density};

\def\xzero{0}
\def\xone{8.4}
\def\xthreshold{0.55} 
\def\xtree{1.0}
\def\xwheel{2.2}
\def\xdense{6.3} 

\draw (\xzero,0.08) -- (\xzero,-0.08) node[below] {$0$};
\draw (\xthreshold,0.08) -- (\xthreshold,-0.08) node[below] {$1/n$};
\draw (\xone,0.08) -- (\xone,-0.08) node[below] {$1$};
\draw (\xdense,0.08) -- (\xdense,-0.08) node[below] {$0.75$};

\fill (\xtree,0) circle (1.6pt);
\fill (\xwheel,0) circle (1.6pt);
\draw (\xtree,0) -- (\xtree,0.18)
  node[above,align=center] {trees\\[0.5mm]$2/n$};
\draw (\xwheel,0) -- (\xwheel,0.18)
  node[above,align=center] {wheels\\[0.5mm]$4/n$};

\draw (\xdense,0) -- (\xdense,0.18);
\draw (\xone,0) -- (\xone,0.18);
\draw[<->, line width=0.6pt] (\xdense,0.16) -- (\xone,0.16);
\node[above,align=center] at ({(\xdense+\xone)/2},0.2)
  {dense graphs};

\draw[<->, line width=0.6pt]
  (\xthreshold,-0.9) -- (\xone,-0.9);
\node[below,align=center] at ({(\xthreshold+\xone)/2},-1)
  {threshold graphs (established in this work)};
\end{tikzpicture}

\vspace{-0.6em}
\captionof{figure}{Some known globally synchronizing graphs ordered by density.}
\label{fig:density}
\vspace{-0.8em}
\end{center}
In this work, we add a classical graph class, namely threshold graphs, to this picture.
More precisely, we prove that every connected threshold graph is second-order globally synchronizing. This implies it is also globally synchronizing.
More strikingly, threshold graphs can realize essentially any edge density and any minimum degree (as we will show in Section~\ref{sec:threshold}).
To the best of our knowledge, this is the first globally synchronizing graph class that simultaneously spans the full spectrum from sparse to dense networks and realizes arbitrary minimum degrees.

Our main result is the following:
\begin{theorem}\label{main}
Every connected threshold graph is second-order globally synchronizing. In particular, every connected threshold graph is globally synchronizing.
\end{theorem}

The precise definition and equivalent characterizations of threshold graphs are given in Section~\ref{sec:threshold}. Here we provide a brief introduction. A threshold graph can be constructed from a single vertex by iteratively adding each new vertex in one of two ways: either as an \emph{isolated vertex}, adjacent to no existing vertex, or as a \emph{dominating vertex}, adjacent to all existing vertices. This generative process is uniquely encoded by a binary sequence of length \( n-1 \), where \texttt{0} denotes an isolated vertex and \texttt{1} a dominating one.
See Figure~\ref{fig:threshold} for examples.

Interestingly, threshold graphs interpolate between two familiar globally
synchronizing examples: star graphs at the sparse end and complete graphs at the
dense end. This was one of our initial motivations for studying this class. While
the intermediate threshold graphs are neither tree-like nor close to complete in
general, {their neighborhoods are highly nested. Intuitively, this nested
structure enforces synchronization by propagating phase alignment along the
construction sequence, thereby ruling out stable nonsynchronous configurations.}
\begin{figure}[H]
    \centering
    \includegraphics[width=\linewidth]{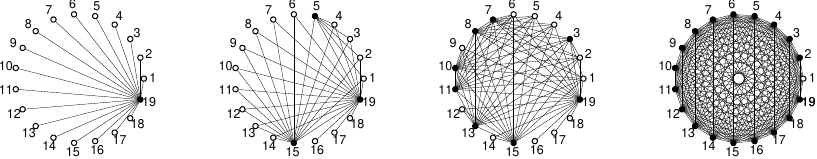}
    \captionsetup{width=\textwidth}
    \caption{Four connected threshold graphs on $19$ vertices, from the sparsest (star graph) to the densest (complete graph). Node labels indicate the vertex addition order in the construction, with vertex $1$ as the initial vertex. White and black circles denote isolated and dominating vertices (bit $0$ and bit $1$), and the initial vertex $1$ is colored in white. Codes from left to right: \texttt{000000000000000001}, \texttt{000100000000010001}, \texttt{010001101101010001}, \texttt{111111111111111111}.}
    \label{fig:threshold}
\end{figure}

To the best of our knowledge, this is the first result establishing global
synchronization for threshold graphs. We believe that this result provides new
insights into the structure of globally synchronizing graphs and introduces a new proof strategy, as summarized below.
\newline

\noindent
(1) Our result shows that global synchronization is not merely a consequence of
high edge density or large minimum degree, but is strongly influenced by the
underlying structure of the graph. Much of the existing work has focused on the
dense regime, where global synchronization is guaranteed by imposing a very large
minimum degree. While examples such as trees and wheel graphs already indicate
that density alone is not the decisive factor, these examples represent only
special density regimes, namely the very sparse regime. In contrast, threshold
graphs provide a structurally rich class of globally synchronizing graphs across
essentially the entire density range. As shown in Section~\ref{sec:threshold},
they span edge densities from \(2/n\) to \(1\), and realize all possible minimum
degrees from \(1\) to \(n-1\). Equivalently, for any prescribed edge density in
this range, there exists a globally synchronizing graph with that density. The
same holds for the minimum degree.

It is worth noting that the recent work on global synchronization of expander
graphs~\cite{abdalla2026expander} also supports the view that globally
synchronizing graphs need not be extremely dense, and that the structure of the
graph matters. This makes it natural to compare our result with expander-based
results.
At a high level, the global synchronization results for expander graphs rely on
the well-known \textit{expander mixing lemma}. Roughly speaking, it is this
mixing property that rules out stable equilibria in which the oscillators are
well spread over the circle: for any two vertex sets, the number of edges between
them is close to what one would expect if edges were distributed at random.
However, our theorem is not covered by~\cite{abdalla2026expander}, since
threshold graphs are typically poor expanders. Their nested-neighborhood
structure creates large variations in the number of edges between different
pairs of vertex sets. Consequently, for some vertex sets \(S,T\), the edge count
\(e(S,T)\) can deviate substantially from the ``expected'' value predicted by the
average degree. For example, consider the star graph, which is a connected
threshold graph. Let \(S\) consist of the centre and let \(T\) be the set of
leaves. Then
\[
e(S,T)=n-1,
\]
whereas the value predicted by the average degree \(\bar d=2(n-1)/n\) is only of constant order:
\[
\frac{\bar d}{n}|S||T|
=
\frac{2(n-1)}{n^2}\cdot (n-1)
=
\frac{2(n-1)^2}{n^2}
=O(1).
\]

\noindent
(2) Our result shows that globally synchronizing graphs can have high heterogeneous degree sequence. It is interesting to note that the minimum-degree conditions necessarily enforce a high degree of homogeneity in the degree sequence. In contrast, we show in Section~\ref{discuss} that the degree sequences of threshold graphs are extremal: among all graphs with the same number of edges, they majorize every other degree sequence. Thus threshold graphs provide a class of globally synchronizing graphs with highly heterogeneous degree distributions, standing in sharp contrast to the degree homogeneity forced by large minimum-degree conditions.
\newline

\noindent
(3) Many existing approaches to global synchronization rely on the so-called \textit{half-circle lemma}, which states that for connected graphs, if all oscillators are contained in an open half-circle, then the dynamics contract toward synchrony. The remaining task is then to show that every second-order stationary point must lie within such a half-circle. One core proof difficulty is that it is unclear how to prove that threshold graphs are globally synchronizing using the half-circle lemma. In this work, we bypass the half-circle approach by exploiting the phasor geometry of local structures in the graph through the first- and second-order stationarity conditions, and by introducing an inductive argument that runs backwards along the construction sequence of threshold graphs.
From a bird's eye view, global synchronization of threshold graphs arises from the presence of local symmetries. Some of these symmetries come from the graph structure itself: as we show in Lemma~\ref{sec:closed-twins}, closed twins must synchronize at any second-order stationary point.
There are also symmetries that do not originate purely from combinatorial structure, that is, certain configurations exhibit a form of geometric symmetry. As we will show in Section~\ref{sec:sync-pendant}, in the configurations we call synchronous pendants, pairs of vertices that are not structurally symmetric nonetheless become geometrically equivalent at second-order stationary points and hence synchronize.

\subsection{Structure of the paper}
The preliminaries on threshold graphs and the Kuramoto model are presented in Sections~\ref{sec:threshold} and~\ref{pre-kura}, respectively.
The basic geometric tools, referred to as \emph{phasor geometry} or \emph{vector geometry}, are introduced in Section~\ref{sec:geo-eq}.
The key lemmas, termed \emph{local synchronization primitives}, which underpin our main proof, are presented in Section~\ref{Primitives}.
The proof of the main result is given in Section~\ref{main_pf}.
Finally, Section~\ref{summary} concludes the paper with a discussion and open questions.

\subsection{Notation}
Throughout this work, all graphs are assumed to be connected and simple, i.e., without self-loops or weighted edges. We write $u \sim v$ to indicate that vertices $u$ and $v$ are adjacent in a graph, and $u \not\sim v$ otherwise. For a vertex $i$ and a subset $Q$ of vertices, we denote $N_Q(i) := \{ q \in Q : q \sim i \}$, the set of neighbors of $i$ that lie in $Q$. We use $\sqcup$ to denote the disjoint union of sets. For a vertex $u\in V$, let $N(u):=\{v\in V:\{u,v\}\in E\}$ denote the open neighborhood of $u$, and let $N[u]:=N(u)\cup\{u\}$ denote the closed neighborhood of $u$.
We denote by $\mathbb{S}^1 := \{ \boldsymbol x \in \mathbb{R}^2 : \|\boldsymbol x\| = 1 \}$ the unit circle. For vectors $\boldsymbol x, \boldsymbol y \in \mathbb{R}^2$, we write $\boldsymbol x \Uparrow \boldsymbol y$ to indicate that $\boldsymbol x$ and $\boldsymbol y$ are nonzero and positively collinear, i.e., $\boldsymbol x = \lambda \boldsymbol y$ for some $\lambda > 0$.
Following the convention in dynamical systems, we refer to the variable $\boldsymbol{\theta} = (\theta_1,\ldots,\theta_n)$ of the Kuramoto model~(1.2) as the \emph{state}, and call the solution $\boldsymbol{\theta}(t)$ with initial condition $\boldsymbol{\theta}(0)$ a \emph{trajectory} starting from $\boldsymbol{\theta}(0)$.
\section{Preliminaries on threshold graphs and their density}\label{sec:threshold}
Threshold graphs form a classical and well-studied graph class, introduced by Chvátal and Hammer~\cite{ChvatalHammer1977} in connection with set-packing and integer programming.
Since then, they have appeared in several areas of discrete mathematics, including graphical degree sequences~\cite{CH73}, spectral graph theory, and spanning tree enumeration~\cite{HammerKelmans1996}.
They have also found applications in resource allocation, cyclic scheduling, manpower planning, and complexity theory \cite{MahadevPeled1995}.

\subsection{Definition and Characterizations}
\begin{definition}\label{def:threshold}
A graph $G$ is called a \emph{threshold graph} if it can be constructed from
the one-vertex graph by repeatedly applying one of the following two
operations:
\begin{enumerate}
    \item[(i)] add an \emph{isolated} vertex, adjacent to no existing vertex;
    \item[(ii)] add a \emph{dominating} vertex, adjacent to all existing
    vertices.
\end{enumerate}
\end{definition}

The resulting graph on vertices $\{1,\dots,n\}$ is uniquely encoded by the
binary sequence $(b_2,\dots,b_n)$, where $b_i=0$ if vertex $i$ is added as
isolated and $b_i=1$ if it is added as dominating node.
Threshold graphs also admit several equivalent characterizations \cite{ChvatalHammer1977}.
\begin{theorem}\label{characterization}
Let $G = (V,E)$ be a graph. The following statements are equivalent:
\begin{enumerate}
    \item[(i)] $G$ is a threshold graph.
    \item[(ii)] $G$ is obtained by iteratively adding isolated or dominating vertices to the one-vertex graph.
    \item[(iii)] $G$ admits a weight function $w : V \to \mathbb{R}$ and a threshold $t \in \mathbb{R}$ such that
\[
\{u,v\} \in E
\;\Longleftrightarrow\;
w(u) + w(v) \ge t,
\qquad \forall\, u \ne v .
\]

    \item[(iv)] $G$ is $\{P_4, C_4, 2K_2\}$-free (see Fig.~\ref{fig:threshold-forbidden}).
    \item[(v)] 
    For any distinct $u,v \in V$, $
        N(u) \subseteq N[v] \quad \text{or} \quad N(v) \subseteq N[u].$
\end{enumerate}
\end{theorem}
\begin{remark}
    It is interesting to note that in Section~VI of \cite{canale2022weighted}, it is suggested that a forbidden-subgraph characterization of globally synchronizing graphs may be implausible in general. Restated in graph-theoretic terms, Theorem~\ref{main} provides a complementary positive example: every graph with no induced subgraph isomorphic to \(P_4\), \(C_4\), or \(2K_2\) is globally synchronizing.
\end{remark}
\begin{figure}[H]
  \centering
  \tikzset{
    vtx/.style={circle, fill=black, inner sep=1.6pt},
    every label/.style={font=\small}
  }

  \begin{subfigure}[b]{0.28\textwidth}
    \centering
    \begin{tikzpicture}[scale=1]
      \node[vtx,label=below:$1$] (a) at (0,0) {};
      \node[vtx,label=below:$2$] (b) at (0.9,0) {};
      \node[vtx,label=below:$3$] (c) at (1.8,0) {};
      \node[vtx,label=below:$4$] (d) at (2.7,0) {};
      \draw (a)--(b)--(c)--(d);
    \end{tikzpicture}
    \caption{$P_4$}
  \end{subfigure}
  \hfill
  \begin{subfigure}[b]{0.28\textwidth}
    \centering
    \begin{tikzpicture}[scale=1]
      \node[vtx,label=below:$1$] (a) at (0,0) {};
      \node[vtx,label=below:$2$] (b) at (1.3,0) {};
      \node[vtx,label=above:$3$] (c) at (1.3,1.0) {};
      \node[vtx,label=above:$4$] (d) at (0,1.0) {};
      \draw (a)--(b)--(c)--(d)--(a);
    \end{tikzpicture}
    \caption{$C_4$}
  \end{subfigure}
  \hfill
  \begin{subfigure}[b]{0.28\textwidth}
    \centering
    \begin{tikzpicture}[scale=1]
      \node[vtx,label=below:$1$] (a) at (0,0) {};
      \node[vtx,label=below:$2$] (b) at (0.9,0) {};
      \node[vtx,label=below:$3$] (c) at (1.9,0) {};
      \node[vtx,label=below:$4$] (d) at (2.8,0) {};
      \draw (a)--(b);
      \draw (c)--(d);
    \end{tikzpicture}
    \caption{$2K_2$}
  \end{subfigure}
  \caption{Forbidden induced subgraphs for threshold graphs.}
  \label{fig:threshold-forbidden}
\end{figure}
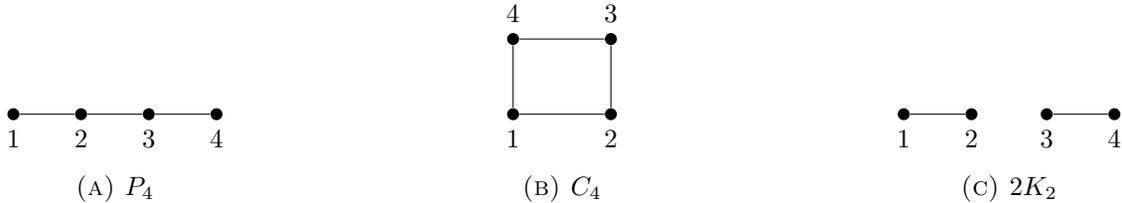

\begin{remark}
Threshold graphs appeared in the synchronization literature long before their role in the Kuramoto model was considered. In particular, as early as 1977, Ausiello, Messina, and Protasi~\cite{Ausiello1977} introduced a class of graphs (later recognized as threshold graphs) in the study of synchronization primitives in distributed computation. Their work provided a graph-theoretic characterization of so-called PV$_c$-definable graphs, which encode combinatorial conditions under which simple local update rules lead to global agreement.
\end{remark}

\subsection{Threshold Graphs Realize All Densities}\label{discuss}
The edge density of a graph $G$ on $n$ vertices takes values in
\begin{equation}\label{all}
\mathrm{density}(G)\in
\left\{\, \frac{m}{\binom{n}{2}} \;\middle|\; m \in \{0,\ldots,\binom{n}{2}\} \,\right\}.
\end{equation}
Moreover, for connected graphs, the edge density takes values corresponding to
\[
m \in \{n-1,\ldots,\binom{n}{2}\}.
\]

The following lemma shows that every admissible edge density can be realized by a
threshold graph.

\begin{lemma}
For any pair of integers $n>0$ and $0 \le m \le \binom{n}{2}$, there is a threshold
graph with $n$ vertices and $m$ edges. If $m \ge n-1$, there is a connected threshold
graph with $n$ vertices and $m$ edges.
\end{lemma}

\begin{proof}
The proof is by induction over $n$. For the base case, note that the only
graph with $n=1$ vertex has $0=m=\binom{n}{2}$ edges and is connected by definition.
For $n>1$, we consider two cases:
\begin{itemize}
\item
If $m<n-1$, the induction hypothesis implies that there is a threshold
graph with $n-1\ge 1$ vertices and $m<n-1\le \binom{n-1}{2}$ edges to which an
isolated vertex can be added to obtain a threshold graph with $n$ vertices
and $m$ edges.
\item
Otherwise, the induction hypothesis implies that there is a threshold graph
with $m-(n-1)\le \binom{n-1}{2}$ edges to which a universal vertex can be added
to obtain a connected threshold graph with $n$ vertices and $m$ edges.
\end{itemize}
\end{proof}

\subsection{Degree Sequence of Threshold Graphs}
One way to compare how spread out the degree sequences of two graphs are is through majorization.
Let $G$ and $G'$ be two graphs with the same edge density, and let
\[
\mathbf d(G) = (d_1,\dots,d_n), \qquad 
\mathbf d(G') = (d'_1,\dots,d'_n)
\]
denote their degree sequences arranged in nonincreasing order.
We say that $\mathbf d(G)$ \emph{majorizes} $\mathbf d(G')$, and write $\mathbf d(G) \succ \mathbf d(G')$, if
\[ \sum_{i=1}^k d_i \ge \sum_{i=1}^k d'_i \quad \text{for all } k=1,\dots,n-1, \]
and
\[ \sum_{i=1}^n d_i = \sum_{i=1}^n d'_i. \]
The following is a classical result on degree sequence (called a \emph{threshold sequence}) of threshold graphs.

\begin{theorem}[\cite{peled1989polytope}]\label{major}
A degree sequence $\mathbf d$ is a threshold sequence if and only if it is not strictly majorized by any other graphical degree sequence.
\end{theorem}
This theorem formalizes the intuition that threshold graphs are extremal, in the sense that their degree mass is concentrated as much as possible on high-degree vertices.

\subsection{Threshold Graphs Realize All Minimum Degrees}
\begin{theorem}
For every \(n\ge 2\) and every \(k\in\{1,\ldots,n-1\}\), there exists a
connected threshold graph on \(n\) vertices with minimum degree \(k\).
\end{theorem}

\begin{proof}
Consider the threshold graph with construction sequence
\[
0^{\,n-k}1^{\,k}.
\]
The first \(n-k\) vertices are added as isolated vertices, and the last \(k\)
vertices are added as dominating vertices. Hence each of the first \(n-k\)
vertices has degree \(k\), while each of the last \(k\) vertices has degree
\(n-1\). Therefore the minimum degree is \(k\).
\end{proof}
\section{Preliminaries on the Kuramoto Model}
\label{pre-kura}

\subsection{Equilibria and Local Stability}
By the definition, equilibrium points of a dynamical system are the states where the system remains at rest, meaning that the first-order derivative of the state vanishes \cite{khalil2002nonlinear}.
Thus setting the right-hand side of \eqref{kuramoto} to zero provides the characterization of an equilibrium point of the Kuramoto model.

\begin{definition}[Equilibria of \eqref{kuramoto}]\label{equili}
We call $\boldsymbol\theta \in \mathbb{R}^n$ an \emph{equilibrium} of the Kuramoto model~\eqref{kuramoto} 
if it satisfies
\begin{equation}\label{gradient-vanish}
    \sum_{j=1}^n \boldsymbol{A}_{ij}\sin(\theta_j-\theta_i)=0,
    \quad \forall i\in[n].
\end{equation}
\end{definition}

An equilibrium point is said to be stable if sufficiently small perturbations around it remain small for all future time~\cite{khalil2002nonlinear}. 
To assess the stability of a given equilibrium point, one typically studies the linearization of the system at that point, which is characterized by the Jacobian matrix.

\begin{proposition}[Stable equilibrium of \eqref{kuramoto}]\label{second-order1}
	An equilibrium point \(\boldsymbol{\theta}\in\mathbb{R}^n\) of the Kuramoto model~\eqref{kuramoto} is stable if the Jacobian matrix \(\boldsymbol{\mathbf{J}}\), defined as: \[\boldsymbol{\mathbf{J}}_{ij} = \begin{cases}
    -\sum_{j=1}^n \boldsymbol{A}_{ij} \cos(\theta_i - \theta_j), & \text{if } i = j, \\[8pt]
    \boldsymbol{A}_{ij} \cos(\theta_i - \theta_j), & \text{if } i \neq j, \end{cases}\] is negative semi-definite.
\end{proposition}

\begin{proof}
Let $F_i(\boldsymbol\theta)=\sum_{j=1}^n \boldsymbol{A}_{ij}\sin(\theta_j-\theta_i)$ denote the right-hand side of the Kuramoto equation. Then \[\frac{\partial F_i}{\partial \theta_i}=-\sum_{j=1}^n \boldsymbol{A}_{ij}\cos(\theta_j-\theta_i),\] while for $i\neq j$, \[\frac{\partial F_i}{\partial \theta_j}=\boldsymbol{A}_{ij}\cos(\theta_j-\theta_i).\] Thus the Jacobian matrix has the stated entries.
\end{proof}

\begin{remark}
Since the homogeneous Kuramoto model is invariant under global phase shifts, the Jacobian matrix $\mathbf J$ always satisfies $\mathbf J \mathbf 1=0$. Thus every equilibrium has at least one zero eigenvalue, corresponding to the rotational symmetry of the model. We call an equilibrium \emph{degenerate} if $\mathbf J$ has more than one zero eigenvalue. The extreme case where all eigenvalues are zero is known as \emph{complete degeneracy}; The existence and stability of such equilibria were characterized in \cite{sclosa2024completely}. In the present work, we do not need to analyze degenerate equilibria separately, since they will be excluded by showing that every equilibrium with negative semidefinite Jacobian is fully synchronized, modulo global phase shifts.
\end{remark}

\subsection{Stationary Points and Local Optimality}
\begin{definition}[First-order stationary points of \eqref{E}]\label{stationaryy}
We call $\boldsymbol\theta\in\mathbb{R}^n$ a \emph{first-order stationary point} of the energy function~\eqref{E}
if it satisfies $\nabla E(\boldsymbol\theta) = 0$, namely
\begin{equation}\label{stationary}
    \sum_{j=1}^n \boldsymbol{A}_{ij}\sin(\theta_j-\theta_i)=0,
    \quad \forall i\in[n].
\end{equation}
\end{definition}

\begin{proposition}[Hessian $\boldsymbol{H}$ of \eqref{E}]\label{hessian}
For the Kuramoto energy~\eqref{E}, the Hessian has entries
\[
\boldsymbol{H}_{ij}:=(\nabla^2 E(\boldsymbol{\theta}))_{ij} =
\begin{cases}
    \sum_{j=1}^n \boldsymbol{A}_{ij} \cos(\theta_i - \theta_j), & \text{if } i = j, \\[8pt]
    -\boldsymbol{A}_{ij} \cos(\theta_i - \theta_j), & \text{if } i \neq j, \end{cases}
\]
\end{proposition}

\begin{proof}
Recall that $E(\boldsymbol\theta)=\frac12\sum_{i,j=1}^n \boldsymbol{A}_{ij}\bigl(1-\cos(\theta_i-\theta_j)\bigr)$. Differentiating with respect to $\theta_i$ gives $\frac{\partial E}{\partial \theta_i}=
\sum_{j=1}^n \boldsymbol{A}_{ij}\sin(\theta_i-\theta_j)$. Differentiating once more, we obtain
\[
\frac{\partial^2 E}{\partial \theta_i^2}
=
\sum_{j=1}^n \boldsymbol{A}_{ij}\cos(\theta_i-\theta_j),
\]
while for $i\neq j$,
\[
\frac{\partial^2 E}{\partial \theta_i\partial \theta_j}
=
- \boldsymbol{A}_{ij}\cos(\theta_i-\theta_j).
\]
Hence the Hessian matrix $\mathbf H=\nabla^2 E(\boldsymbol\theta)$ has entries as claimed.
\end{proof}

\begin{definition}[Second-order stationary points of \eqref{E}]
A configuration $\boldsymbol\theta\in\mathbb R^n$ is called a \emph{second-order stationary point} (SOSP) of~\eqref{E} if $\nabla E(\boldsymbol\theta) = 0$, and the Hessian $\boldsymbol{H}$ is positive semidefinite.
\end{definition}

\subsection{Correspondence Between the Two Sets of Concepts}
\begin{proposition}
    A configuration $\boldsymbol{\theta}\in\mathbb R^n$ is an equilibrium of the Kuramoto model \eqref{kuramoto} if and only if it is a first-order stationary point of the Kuramoto energy \eqref{E}.
\end{proposition}
\begin{proof}
    This is clear, since both are characterized by $$\sum_{j=1}^n \boldsymbol A_{ij} \sin(\theta_j - \theta_i) = 0,\qquad \forall i \in [n].$$
\end{proof}

Moreover, since \eqref{E} is real analytic, an equilibrium of \eqref{kuramoto} is stable if and only if it is a local minimum of \eqref{E}, modulo the global rotation symmetry.
Such an equivalence was established in the celebrated work~\cite{Absil2006StableEquilibrium}.

\begin{definition}[Local minimum of the energy function]
A configuration $\boldsymbol\theta \in \mathbb{R}^n$ is called a \emph{local minimum} of the energy function $E(\boldsymbol\theta)$ if there exists $\varepsilon > 0$ such that
\[
E(\boldsymbol\theta) \leq E(\boldsymbol\theta'), \quad \forall \boldsymbol\theta' \in \mathbb{R}^n \ \text{with}\ \|\boldsymbol\theta'-\boldsymbol\theta\| < \varepsilon.
\]
\end{definition}

Therefore, the relations between stable equilibria, second-order stationary points and local minima can be summarized in Figure~\ref{equivalence-gradient}.

\begin{figure}[H]
\centering
\begin{tikzpicture}
  \draw[thick] (0,0) ellipse (3cm and 1.8cm);
  \node at (0.2, 1.3) {\small SOSP of \eqref{E}};
  \draw[thick, fill=gray!20] (0,0) ellipse (2.2cm and 0.8cm);
\node at (0,0) {\small \parbox{5.5cm}{\centering Local minima of \eqref{E} \\  = stable equilibria of \eqref{kuramoto}}};
\end{tikzpicture}
\caption{Relation between different concepts.}
\label{equivalence-gradient}
\end{figure}
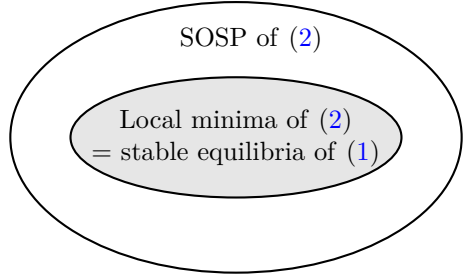
\section{Vector Geometry of Equilibria}\label{sec:geo-eq}
In this section, we recast the algebraic conditions introduced in Section~\ref{pre-kura} into geometric ones, namely combinatorial relations among vectors associated with the nodes. By doing so, the analysis of global synchronization on threshold graphs reduces to reasoning purely about geometric relations between the unit vectors corresponding to the angles in~$\boldsymbol{\theta}$.

\subsection{Vector-labeled Graph}
Given a state \( \boldsymbol{\theta} \in \mathbb{R}^n \) of the Kuramoto model on graph $G=(V,E)$. Associate each node \( i \in V \) its corresponding unit vector
\[
\boldsymbol{v}_i := (\cos \theta_i, \sin \theta_i) \in \mathbb{R}^2.
\]
This is called a \emph{phasor}, following the terminology from Yoshiki Kuramoto \cite{Kuramoto1975}.
The object \textit{vector-labeled graph} \( \left( G, \{ \boldsymbol{v}_i \}_{i \in V} \right) \) encodes \( \boldsymbol{\theta} \) combinatorially and geometrically.

\begin{figure}[H]
\centering
\begin{tikzpicture}[scale=1.8]
\draw[thick, black] (0,0) circle (1);
\fill[gray!20] (0,0) -- (0.25,0) arc (0:60:0.25) -- cycle;
\draw[thick, -stealth] (0,0) -- (60:1) 
    node[fill=black,circle,inner sep=1.5pt,label=above:$\boldsymbol{v}_i$]{};
\fill[black] (0,0) circle (0.5pt);
\draw[dotted, thick, black] (0,0) -- (1,0);
\node at (30:0.35) {$\theta_i$};
\end{tikzpicture}
\caption{Representation of $\boldsymbol{v}_i$ with angle $\theta_i$.}
\end{figure}

\subsection{Alignment and Equilibrium}\label{Alignment-and-equilibrium}
Given a vector-labeled graph $(G, \{\boldsymbol{v}_i\}_{i\in V})$ associated with the state~$\boldsymbol{\theta}$, 
we reformulate the equilibrium condition~\eqref{gradient-vanish} (equivalently~\eqref{stationary}) in geometric terms.

\begin{lemma}[Equilibrium condition rephrased geometrically]\label{eq-vec}
A state \( \boldsymbol{\theta} \) is an equilibrium of the Kuramoto model~\eqref{kuramoto} on $G=(V,E)$ with adjacency $\boldsymbol{A}$ {if and only if}, for each \( i \in V \), 
\begin{equation}\label{formula:eq-ec}
\sum_{j\in N(i)} \boldsymbol{v}_j = \mu_i \boldsymbol{v}_i\quad\text{where}\quad \mu_i=\sum_{j\in N(i)}\cos(\theta_j-\theta_i).
\end{equation} where $\boldsymbol v_i:=(\cos\theta_i,\sin\theta_i)$.
\end{lemma}
The proof can be found in \cite{monzon2005global}, we present here for completeness.
\begin{proof}[Proof of Lemma \ref{eq-vec}]
The condition~\eqref{formula:eq-ec} is equivalently written as
\begin{equation}\label{equivalent-eq}
    \sum_{j\in N(i)} e^{I\theta_j} = \mu_i e^{I\theta_i}.
\end{equation}
Indeed, rewrite all vectors \( \boldsymbol{v}_i \) in (\ref{formula:eq-ec}) in terms of their Cartesian components, yielding  
\[ \sum_{j\in N(i)} (\cos\theta_j,\sin\theta_j) = \mu_i(\cos\theta_i,\sin\theta_i). \]  
Equivalently,
\begin{equation}\label{eq:cartesian-balance}
\left\{
\begin{aligned}
    \sum_{j\in N(i)}\cos\theta_j &= \mu_i\cos\theta_i, \\
    \sum_{j\in N(i)}\sin\theta_j &= \mu_i\sin\theta_i.
\end{aligned}
\right.
\end{equation}  
The condition \eqref{eq:cartesian-balance} is equivalent to the equality between the following two complex numbers:  
\[
\sum_{j\in N(i)} \cos\theta_j + I\sum_{j\in N(i)}\sin\theta_j = \mu_i\cos\theta_i + I \mu_i \sin\theta_i,
\]  
which can be compactly rewritten as \eqref{equivalent-eq} using the Euler formula $e^{I\theta_i}=\cos\theta_i+I\sin\theta_i$.

Now that we have derived the equivalence between conditions (\ref{formula:eq-ec}) and (\ref{equivalent-eq}), our goal reduces to proving the equivalence between (\ref{equivalent-eq}) and (\ref{gradient-vanish}), namely
\[ \forall i\in V,\qquad \sum_{j\in N(i)}e^{I\theta_j} = \mu_i e^{I\theta_i} \Longleftrightarrow \sum_{j\in N(i)}\sin(\theta_j-\theta_i) = 0. \]
The implication from the left-hand side to the right-hand side is established by showing that
\[
I\sum_{j\in N(i)} \sin(\theta_j-\theta_i ) + \sum_{j\in N(i)}\cos(\theta_j-\theta_i) \in \mathbb{R}.
\]
This can be verified by a direct computation: $$\begin{aligned}
        & I\sum_{j\in N(i)} \sin(\theta_j-\theta_i) + \sum_{j\in N(i)}\cos(\theta_j-\theta_i)\\
        &= \sum_{j\in N(i)}e^{I(\theta_j-\theta_i)}\\
        &= \frac{\sum_{j\in N(i)} e^{I\theta_j}}{e^{I\theta_i}}\\
        &= \mu_i.
    \end{aligned}$$
To prove the other direction, we construct for every node $i\in V$, such a \(\mu_i\) that satisfies the left-hand side. 
Dividing both sides of (\ref{equivalent-eq}) by \( e^{I\theta_i} \) (noting that \( e^{I\theta_i} \neq 0 \)), we obtain \[ \frac{\sum_{j\in N(i)} e^{I\theta_j}}{e^{I\theta_i}} = \mu_i. \]
Compute $$
\begin{aligned}
    \frac{ \sum_{j\in N(i)} e^{I\theta_j}}{e^{I\theta_i}} &= \sum_{j\in N(i)} e^{I(\theta_j - \theta_i)}\\
    &= \sum_{j\in N(i)} (\cos(\theta_j - \theta_i) + I\sin(\theta_j - \theta_i))\\
    &= \sum_{j\in N(i)} \cos(\theta_j - \theta_i)
\end{aligned}
$$
where the last equality is due to the fact $\sum_{j\in N(i)}\sin(\theta_j-\theta_i)=0$.
Thus $$\sum_{j\in N(i)}e^{I\theta_j} = \mu_i e^{I\theta_i}$$ holds for every $i\in V$ by letting $\mu_i=\sum_{j\in N(i)}\cos(\theta_j-\theta_i)$.
\end{proof}

\begin{remark}[Normalization and local coherence]
One can rewrite~\eqref{eq-vec} into
\[
\boldsymbol{v}_i = \frac{1}{\mu_i} \sum_{j\in N(i)} \boldsymbol{v}_j, 
\quad \text{where } \mu_i = \left\| \sum_{j\in N(i)} \boldsymbol{v}_j \right\| 
= \sum_{j\in N(i)} \cos(\theta_j - \theta_i).
\]

The expression describes a geometric condition of equilibrium:
\textit{At equilibrium, each phasor aligns with the direction of the sum of its neighbors' vectors and is scaled to unit length.}
Although introduced as a normalization factor, the quantity \( \mu_i \) carries physical meanings.  
If the neighbors are highly coherent, the sum vector is large, requiring a larger \( \mu_i \) to normalize \( \boldsymbol{v}_i \) to unit length; if the directions are dispersed, \( \mu_i \) is small.  
In this sense, \( \mu_i \) measures how strongly the phasors of node \( i \)'s neighbors align directionally.  
Indeed, defining the local order parameter $R_i := \frac{1}{\deg(i)} \left| \sum_{j \in N(i)} e^{I\theta_j} \right|,$
we have the relation \( \mu_i = \deg(i) \cdot R_i \), showing that \( \mu_i \) is proportional to the standard notion of local phase coherence.
\end{remark}

\begin{remark}[Equilibrium and \textit{force}]
The Kuramoto model \eqref{kuramoto} admits a natural physical interpretation using the concept of \textit{force}.  
Write each phasor \( \boldsymbol{v}_i \) in complex form as \( e^{I\theta_i} \).  
Then, intuitively, each neighboring node \( j \in N(i) \) pulls node \( i \) with a \textit{force} given by $e^{I(\theta_j-\theta_i)}$.
Under this interpretation, equation~\eqref{eq-vec} states that:
    \textit{At equilibrium, the total force acting on node \( i \) sums to a real scalar.}
In other words, the net force is purely radial and contributes only to scaling.  
Due to the unit circle geometry \( \left| \boldsymbol{v}_i \right| = 1 \) of the Kuramoto model, such a scaling force does not induce any motion while a tangential component can cause rotation.
\end{remark}
    
\subsection{Angles and Second-Order Stationarity}\label{sec:angle-stability}
Given an equilibrium $\boldsymbol{\theta}\in\mathbb R^n$ of the Kuramoto model on graph $G=(V,E)$, if there exists a node $i \in V$ whose corresponding $\mu_i < 0$, then $\boldsymbol{\theta}$ cannot be a second-order stationary point of the Kuramoto energy \eqref{E}.
This is because a small perturbation that moves $\boldsymbol{v}_i$ toward the direction of $\sum_{j\in N(i)} \boldsymbol{v}_j$ strictly decreases the energy.

\begin{lemma}\label{stable-eq}
If $\boldsymbol{\theta}\in\mathbb R^n$ is a second-order stationary point of~\eqref{E},
then for every $i\in V$ there exists $\mu_i\ge 0$ such that $$\sum_{j\in N(i)} \boldsymbol{v}_j = \mu_i \boldsymbol{v}_i.$$
\end{lemma}

\begin{remark}
    In particular, either the summation of phasors of neighborhood of $i$ is a zero vector, i.e., $\sum_{j\in N(i)} \boldsymbol{v}_j=\boldsymbol 0$ (i.e., $\mu_i=0$),
or if it is non-zero, it must point towards the same direction as $\boldsymbol{v}_i$, i.e., $$\angle\!\left(\sum_{j\in N(i)} \boldsymbol{v}_j,\, \boldsymbol{v}_i\right)=0.$$
\end{remark}

\begin{proof}[Proof of Lemma \ref{stable-eq}]
By Proposition \ref{hessian}, the diagonal entry of the Hessian is $$\boldsymbol H_{ii} = \sum_{j\in N(i)} \cos(\theta_i-\theta_j).$$
Since $\theta$ is an equilibrium, by Lemma~\ref{eq-vec}, we have
$$\boldsymbol H_{ii} =\sum_{j\in N(i)} e^{I(\theta_j-\theta_i)}
= \mu_i  \in \mathbb R.$$
If $\mu_i<0$, then $\boldsymbol H_{ii}$ is not positive semidefinite,
which contradicts the second-order stationarity of $\boldsymbol{\theta}$.
\end{proof}

\begin{figure}[H]
\centering
\begin{tikzpicture}[scale=1,>=stealth]
\tikzstyle{vi}=[->, thick]
\tikzstyle{sumv}=[->, thick, dashed]
\tikzstyle{feasible}=[line width=6pt, gray!40, cap=round]

\begin{scope}[shift={(-3,1.5)}]
    \draw[thin] (0,0) circle (1);
    \draw[feasible] (0,0) -- (2,0);
    \draw[vi] (0,0) -- (1,0);
    \draw[sumv] (0,0) -- ({1.6*cos(40)},{1.6*sin(40)});
    \node at (0,-1.4) {(A)};
\end{scope}

\begin{scope}[shift={(3,1.5)}]
    \draw[thin] (0,0) circle (1);
    \draw[feasible] (0,0) -- (2,0);
    \draw[vi] (0,0) -- (1,0);
    \draw[sumv] (0,0) -- (-0.9,0);
    \node at (0,-1.4) {(B)};
\end{scope}

\begin{scope}[shift={(-3,-2.2)}]
    \draw[thin] (0,0) circle (1);
    \draw[feasible] (0,0) -- (2,0);
    \draw[vi] (0,0) -- (1,0);
    \draw[sumv] (0,0) -- (1,0);
    \node at (0,-1.4) {(C)};
\end{scope}

\begin{scope}[shift={(3,-2.2)}]
    \draw[thin] (0,0) circle (1);
    \draw[feasible] (0,0) -- (2,0);
    \draw[vi] (0,0) -- (1,0);
    \draw[sumv] (0,0) -- (1.8,0);
    \node at (0,-1.4) {(D)};
\end{scope}

\end{tikzpicture}

        \caption{The solid arrow represents $\boldsymbol{v}_i$, and the dashed arrow represents $\sum_{j\in N(i)} \boldsymbol{v}_j$. For each node $i$, the four panels show all possible relations between $\boldsymbol{v}_i$ and $\sum_{j\in N(i)} \boldsymbol{v}_j$. Among them, case (A) violates the first-order condition, while case (B) violates the second-order stationary condition, since the aggregated phasor $\sum_{j\in N(i)} \boldsymbol{v}_j$ lies outside the feasible (gray) region corresponding to $\mu_i \ge 0$.}
    \label{fig:stability_conditions}\end{figure}
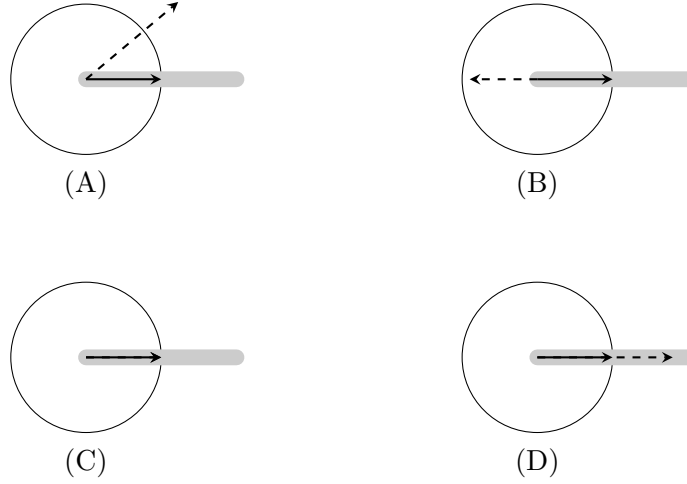

The geometric conditions of second-order stationary points presented here are only the most basic one.
Further necessary geometric conditions will be established in subsequent
sections, and these will be essential for showing that every connected
threshold graph is globally synchronizing.
\section{Local Synchronization Primitives}\label{Primitives}
\subsection{A Fundamental Geometric Fact}
We begin with a geometric lemma that characterizes all possible relative positions of two unit vectors satisfying a pair of linear relations.
This fact will serve as the local synchronization primitive used in the induction argument in Section \ref{main_pf}.

\begin{lemma}\label{false-twins}
Let $\boldsymbol{v}_a, \boldsymbol{v}_b \in \mathbb{S}^1$.
Suppose there exist a vector $\boldsymbol{q} \in \mathbb{R}^2$ and scalars $\mu_a, \mu_b \in \mathbb{R}$ such that
\begin{equation}\label{geo}
	\boldsymbol{v}_b + \boldsymbol{q} = \mu_a \boldsymbol{v}_a
\qquad\text{and}\qquad
\boldsymbol{v}_a + \boldsymbol{q} = \mu_b \boldsymbol{v}_b .
\end{equation}
Then the positions of $\boldsymbol{v}_a$ and $\boldsymbol{v}_b$ fall into one of the following three cases:
\begin{enumerate}[label=(\arabic*), leftmargin=*, itemsep=0.5em]
    \item $\boldsymbol{v}_a = \boldsymbol{v}_b$ and $\mu_a = \mu_b$;

    \item $\boldsymbol{v}_a = -\boldsymbol{v}_b$, $\mu_a + \mu_b = -2$, and $(\mu_a,\mu_b)\neq (-1,-1)$;

    \item $\mu_a = \mu_b = -1$, and $\boldsymbol{v}_a + \boldsymbol{v}_b + \boldsymbol{q} = 0$.
\end{enumerate}
\end{lemma}

\begin{proof}
According to \eqref{geo}, we have
$$\boldsymbol{q}=\mu_a\boldsymbol{v}_a-\boldsymbol{v}_b=\mu_b\boldsymbol{v}_b-\boldsymbol{v}_a$$
Thus
\begin{equation}\label{vavb}
    (\mu_a + 1)\boldsymbol{v}_a = (\mu_b + 1)\boldsymbol{v}_b.
\end{equation}
If \(\mu_a \neq -1\) and \(\mu_b \neq -1\), then from \eqref{vavb}, it must hold that
\[ \boldsymbol{v}_a \parallel \boldsymbol{v}_b. \]
This leads to case (1) and (2):

(1) If \(\mu_a = \mu_b\neq -1\), then \(\boldsymbol{v}_a = \boldsymbol{v}_b\);

(2) If \(\mu_a + 1 = -(\mu_b + 1)\), then equivalently, we have \(\mu_a + \mu_b = -2\). This implies $\boldsymbol{v}_a = -\boldsymbol{v}_b$.

\noindent
If instead $\mu_a=\mu_b=-1$, then we have $\boldsymbol{q} + \boldsymbol{v}_a +\boldsymbol{v}_b = \boldsymbol{0}$ corresponding to case (3).
\end{proof}

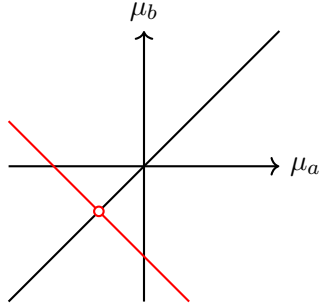
\begin{figure}[H]
\centering
\begin{tikzpicture}[scale=0.6]
    \draw[thick, ->] (-3, 0) -- (3, 0) node[right] {$\mu_a$};
    \draw[thick, ->] (0, -3) -- (0, 3) node[above] {$\mu_b$};

    \draw[thick, black] (-3, -3) -- (3, 3);
    \node[black, above right] at (2,2) {};
    \draw[thick, black] (-3,1) -- (1,-3);
    \node[red, below right] at (1.2,-3.2) {};
    \draw[fill=white, draw=black, thick] (-1,-1) circle (3pt);
\end{tikzpicture}
\caption{The feasible region for $(\mu_a,\mu_b)$ consists of:  
(1) the black line $\mu_a = \mu_b$, excluding the intersection point, corresponding to $\boldsymbol{v}_a = \boldsymbol{v}_b$;  
(2) the red line $\mu_a + \mu_b = -2$, excluding the intersection point, corresponding to the antipodal case $\boldsymbol{v}_a = -\boldsymbol{v}_b$;  
(3) their intersection point $(-1,-1)$, marked by a hollow dot, at which no additional constraint is imposed on the vectors $\boldsymbol{v}_a$ and $\boldsymbol{v}_b$.}
\end{figure}

\begin{corollary}\label{sync_twins}
Under the assumptions of Lemma~\ref{false-twins}, if in addition $\mu_a, \mu_b \ge 0$, then $\boldsymbol{v}_a = \boldsymbol{v}_b$.
\end{corollary}
\begin{proof}
With $\mu_a,\mu_b \ge 0$, Lemma~\ref{false-twins} leaves only case~(1), hence $\boldsymbol{v}_a=\boldsymbol{v}_b$.
\end{proof}

\subsection{Phasor Geometry of Closed Twins}\label{sec:closed-twins}
It is a direct corollary of Lemma \ref{false-twins} and Lemma \ref{stable-eq} that closed twins synchronize at any stable equilibrium $\boldsymbol{\theta}$ of Kuramoto model \ref{kuramoto} on any connected graph. Let us first state the definition of closed twins in the graph theory.
\begin{definition}Let $G=(V,E)$ be a graph. 
Two nodes $i,j \in V$ are called 
\emph{closed twins} if
\begin{equation}\label{st-closed}
N[i] = N[j].\end{equation}
\end{definition}

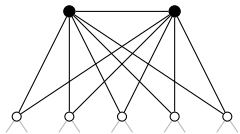
\begin{figure}[H]
\centering
\definecolor{TwinBlue}{RGB}{0,0,0}

\tikzset{
  top/.style={circle,draw,fill=TwinBlue,inner sep=1.5pt},
  bot/.style={circle,draw,fill=white,inner sep=1.2pt},
  edge/.style={thin},
  stub/.style={gray!55, thin}
}

\newcommand{\BottomRow}{
  \foreach \x in {-2,-1,0,1,2}{
    \node[bot] (b\x) at (\x,0) {};
    \draw[stub] (b\x) -- ++(-0.2,-0.3);
    \draw[stub] (b\x) -- ++( 0.2,-0.3);
  }
}
\newcommand{\ConnectAll}[1]{ 
  \foreach \x in {-2,-1,0,1,2}{ \draw[edge] (#1) -- (b\x); }
}

\begin{subfigure}[b]{0.45\textwidth}  
\centering
\begin{tikzpicture}[x=0.7cm,y=0.7cm] 
  \BottomRow
  \node[top] (u) at (-1,2) {};
  \node[top] (v) at ( 1,2) {};
  \ConnectAll{u}
  \ConnectAll{v}
  \draw[edge] (u)--(v); 
\end{tikzpicture}
\end{subfigure}

\caption{The two black nodes form a pair of closed twins.}
\label{fig:twin-two-panels}
\end{figure}

\begin{corollary}\label{cor:twins-sync}
Let $\boldsymbol{\theta}\in\mathbb R^n$ be a second-order stationary point of the energy function $E_G(\boldsymbol{\theta})$. 
If nodes $a$ and $b$ form a pair of closed twins (satisfying \eqref{st-closed}), then $\boldsymbol{v}_a = \boldsymbol{v}_b$.
\end{corollary}

\begin{proof}
From Lemma~\ref{stable-eq}, since $\boldsymbol{\theta}$ is a second-order stationary point,
there exist $\mu_a,\mu_b \ge 0$ such that
\[
\begin{cases}
\displaystyle \sum_{j \in N(a)} \boldsymbol{v}_j = \mu_a \boldsymbol{v}_a, \\[4pt]
\displaystyle \sum_{j \in N(b)} \boldsymbol{v}_j = \mu_b \boldsymbol{v}_b .
\end{cases}
\]

Because $N[a]=N[b]$, we have 
\(
N(a)\setminus\{b\} = N(b)\setminus\{a\},
\)
and we denote this common set by $M$.
Thus,
\[ \begin{cases}
\displaystyle \sum_{j \in M} \boldsymbol{v}_j + \boldsymbol{v}_b = \mu_a \boldsymbol{v}_a, \\[4pt]
\displaystyle \sum_{j \in M} \boldsymbol{v}_j + \boldsymbol{v}_a = \mu_b \boldsymbol{v}_b .
\end{cases} \]

By Corollary~\ref{sync_twins} and the fact that $\mu_a,\mu_b \ge 0$, we conclude that 
$\boldsymbol{v}_a = \boldsymbol{v}_b$.
\end{proof}

Moreover, Corollary \ref{cor:twins-sync} implies immediately that windmill graphs are globally synchronizing.
Indeed, we start from an arbitrary second-order stationary point $\boldsymbol{\theta}$ of $E$ and prove that it is fully synchronized, meaning that $\theta_i=\theta_j$ modulo $2\pi$ for all $i,j\in V$.

\begin{definition}[Windmill graphs]
For integers $p\geq 1$ and $m\geq 2$, the \emph{windmill graph} $W_{p,m}$ is the graph obtained by gluing $m$ copies of $K_p$ along a single common vertex.
\end{definition}

\captionsetup[subfigure]{justification=centering, font=small, skip=2pt}
\begin{figure}[H]
    \centering
        \includegraphics[width=0.3\linewidth]{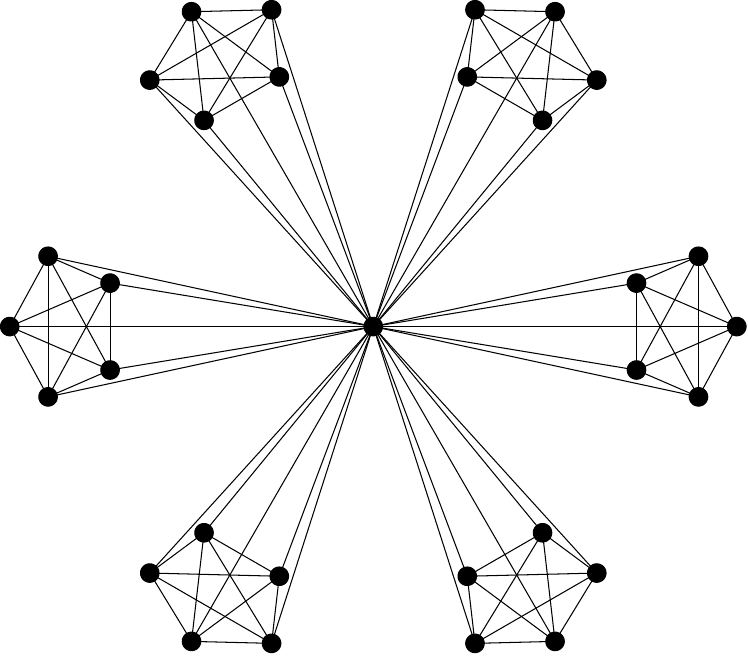}
        \caption{The windmill graph $W_{6,6}$.}
\end{figure}

\begin{theorem}
Every windmill graph $W_{p,m}$ ($p\geq 1$, $m\geq 2$) is second-order globally synchronizing. In particular, it is globally synchronizing.
\end{theorem}

\begin{proof}
Let $G$ be a windmill graph $W_{p,m}$ with central vertex $c$.
Then there exists a partition of the vertex set
\[
V \setminus \{c\} \;=\; \bigsqcup_{l=1}^m S_\ell
\]
such that each induced subgraph $G[S_\ell]$ is a clique, and every vertex in $S_\ell$
is adjacent to $c$.
For each $\ell$, all vertices in $S_\ell$ have identical closed neighborhoods and
therefore form a class of closed twins.
By Corollary~\ref{cor:twins-sync}, all vertices in each $S_\ell$ must share the same phase.
According to Lemma \ref{stable-eq}, we have for $\forall i \in S_l$, that $$\boldsymbol{v}_i =\mu_i\big( \boldsymbol{v}_c+\sum_{j\in S_l,j\neq i}\boldsymbol{v}_j\big),\quad \mu_i\geq 0$$

We distinguish two cases.

\emph{Case 1.} $\mu_i>0$. It is then clear that either $\boldsymbol{v}_c = \boldsymbol{v}_i$ or $\boldsymbol{v}_c = -\boldsymbol{v}_i$ holds.
If $\boldsymbol{v}_c = -\boldsymbol{v}_i$, then $\boldsymbol{\theta}$ is not a second-order stationary point.
Indeed, let $\boldsymbol{x}\in\mathbb{R}^n$ be the vector whose $i$-th component equals $1$ if $i\in S_l$ and $0$ otherwise.
The quadratic form associated with the Hessian matrix $\boldsymbol{H}$ can be computed as
$$
\begin{aligned}
\boldsymbol{x}^\top \boldsymbol{H}\boldsymbol{x} & =\frac{1}{2}\sum_{i=1}^n\sum_{j=1}^n \boldsymbol{A}_{ij}\cos(\theta_i-\theta_j)(x_i-x_j)^2\\
& =
\sum_{i\in S_l}\sum_{j\in V\setminus S_l} \boldsymbol A_{ij}\cos(\theta_i-\theta_j)\\
&=
\sum_{i\in S_l}\cos(\theta_i-\theta_c)
<0,
\end{aligned}
$$
which contradicts the positive semidefiniteness of $\boldsymbol{H}$.

\emph{Case 2.} $\mu_i=0$.
Then necessarily $|S_l|=2$.
In this case, the phasors of the two nodes in $S_l$ are opposite to $\boldsymbol{v}_c$.
By an argument analogous to the one above, the matrix $\boldsymbol{H}$ is not positive semidefinite.
\end{proof}

A set of graph-theoretic closed twins at a second-order stationary point of $E(\boldsymbol{\theta})$ is a special case of the following more general notion, which we call geometrically stable closed twins. We introduce this notion in order to prove that connected threshold graphs are globally synchronizing. Indeed, not all vertices that need to be compared in a threshold graph are closed twins in the usual graph-theoretic sense, so Lemma~\ref{cor:twins-sync} alone only proves the densest extreme case, namely complete graphs. The geometric formulation is needed to handle universal vertices that are not added consecutively in the construction sequence: isolated vertices may be inserted between them, so they need not be graph-theoretic closed twins in the final graph. Nevertheless, at a second-order stationary point they still satisfy the relations in the definition below, and therefore behave as geometrically stable closed twins, and they synchronize.

\begin{definition}[Geometrically closed stable twins]\label{def:geometric_twins}
Let $\boldsymbol{\theta}$ be a state of the Kuramoto model on $G$.
We say that two vertices $i$ and $j$ are \emph{geometrically closed twins}
at $\boldsymbol{\theta}$ if there exist a vector $\boldsymbol q\in\mathbb R^2$
and scalars $\mu_i,\mu_j\in\mathbb R$ such that
\[
    \boldsymbol v_j+\boldsymbol q=\mu_i \boldsymbol v_i,
    \qquad
    \boldsymbol v_i+\boldsymbol q=\mu_j \boldsymbol v_j .
\]
If, in addition, $\mu_i,\mu_j\geq 0$, then we call $i$ and $j$
\emph{geometrically stable closed twins}.
\end{definition}
As shown above, graph-theoretic closed twins at a second-order stationary point form a set of geometrically stable closed twins. Moreover, this phenomenon is not limited to closed twins in the structural sense: as we will show in the next section, there are other configurations in threshold graphs that also give rise
to geometrically stable closed twins.

\subsection{Synchronous Pendant Extension}\label{sec:sync-pendant}

In this section, we present Lemma~\ref{lemma:pendant}, which gives a mechanism by which a group of vertices can form geometrically stable closed twins in the sense of Definition~\ref{def:geometric_twins}, without being closed twins in the graph-theoretic sense. Although these vertices do not share the same closed neighborhood, the differences in their neighborhoods are pendant substructures attached only to the corresponding vertex and synchronized with it. We therefore refer to this result as the \emph{synchronous pendant extension}. Here \textit{pendant} means that they interact with the rest of the graph only through the vertex to which they are attached.
Before that, we establish Lemma~\ref{lem:sync-implies-common-alignment}, which serves as a preparatory step.

\begin{lemma}\label{lem:sync-implies-common-alignment}
Let $G=(V,E)$ be a graph, and let $W \subseteq V$ admit a partition
\[
W = Q \sqcup S \sqcup P,
\]
with the following structural conditions:

(1) Every node in $Q$ has all its neighbors inside $S$, 
    i.e., $N(i) \subseteq S$ for all $i \in Q$;

(2) For each $i \in S$, its neighborhood is 
    $P \sqcup N_Q(i) \sqcup (S \setminus \{i\})$, 
    where $$N_Q(i) = \{\, q \in Q : q \sim i \,\}.$$
Assume further that $\{\boldsymbol{v}_i\}_{i \in V}$ is the vector configuration corresponding to any second-order stationary point $\boldsymbol{\theta}$ of the Kuramoto energy \eqref{E}, and that the induced subgraph on $Q \sqcup S$ is synchronized at $\boldsymbol{\theta}$, namely
$\boldsymbol{v}_i = \boldsymbol{v}$ for all $i \in Q \sqcup S$, for some common phasor $\boldsymbol{v}$.
Then the phasor sum over $P$ satisfies
\[
\sum_{i \in P} \boldsymbol{v}_i \; \Uparrow \; \boldsymbol{v}
\qquad \text{or} \qquad
\sum_{i \in P} \boldsymbol{v}_i = \boldsymbol{0}.
\]
\end{lemma}

\begin{proof}
	For any node $i\in S$, \begin{equation}\label{equa1}
		N(i) = N_Q(i) \sqcup (S\setminus\{i\}) \sqcup P 
	\end{equation} and \begin{equation}\label{equa2}\sum_{j\in N(i)} \boldsymbol{v}_j=\mu_i\boldsymbol{v}_i, \qquad \mu_i\geq 0.\end{equation}
	Plug \eqref{equa1} into \eqref{equa2}, we obtain $$\sum_{j\in N_Q(i) \sqcup(S\setminus \{i\})\sqcup P} \boldsymbol{v}_j=\mu_i\boldsymbol{v}_i, \qquad \mu_i\geq 0$$
	Since $Q \sqcup S$ forms a synchronized group, we have
    \begin{equation}\label{eq:stationarity-S}
        \bigl(|N_Q(i)| + |S| - 1\bigr)\,\boldsymbol{v}_i
    + \sum_{j \in P} \boldsymbol{v}_j
    = \mu_i \boldsymbol{v}_i, \qquad \mu_i \ge 0.
    \end{equation}
	Thus \[ \sum_{j\in P} \boldsymbol{v}_j
    = \bigl(\mu_i - (|N_Q(i)| + |S| - 1)\bigr)\boldsymbol{v}_i,
    \qquad \mu_i \ge 0. \]
    
In fact, $\mu_i - (|N_Q(i)| + |S| - 1) < 0$ does not hold, because in this case there exists a direction along which the 
energy decreases in a neighborhood of $\boldsymbol{\theta}$.
To see this, let $\boldsymbol{x}$ be the $n$-dimensional indicator vector 
of $Q \sqcup S$, i.e., $\boldsymbol x_i = 1$ for $i \in Q \sqcup S$ and $\boldsymbol x_i = 0$ otherwise.
Since $\mu_i-(|N_Q(i)|+|S|-1)<0$, we have $\langle \sum_{j\in P}\boldsymbol{v}_j,\boldsymbol{v}_i \rangle < 0$.
Thus \begin{equation}\label{eq:Hessian-quadratic-form}
\begin{aligned}
\boldsymbol{x}^T H \boldsymbol{x}
    &=\frac{1}{2}\sum_{i=1}^n\sum_{j=1}^n \boldsymbol{A}_{ij}\cos(\theta_i-\theta_j)(x_i-x_j)^2\\
    &= \sum_{i \in Q \sqcup S} \sum_{j \in V\setminus (Q \sqcup S)}\boldsymbol{A}_{ij}\cos(\theta_i - \theta_j)(x_i-x_j)^2 \\
    & = \sum_{i \in S} \sum_{j \in P} \cos(\theta_i - \theta_j)(x_i-x_j)^2 \\
    & = \sum_{i \in S} \sum_{j \in P} \cos(\theta_i - \theta_j)\\
    &= \sum_{i \in S} \sum_{j \in P} \boldsymbol{v}_i \cdot \boldsymbol{v}_j \\
    &= \sum_{i \in S} \boldsymbol{v}_i \cdot \sum_{j \in P} \boldsymbol{v}_j \\
    &= |S| \, \boldsymbol{v}_i \cdot \sum_{j \in P} \boldsymbol{v}_j .
\end{aligned}
\end{equation}
Clearly, if $\mu_i - (|N_Q(i)| + |S| - 1) < 0$, then \eqref{eq:stationarity-S} implies that
\[
\boldsymbol{v}_i \cdot \sum_{j\in P} \boldsymbol{v}_j < 0.
\]
Consequently, the quantity \eqref{eq:Hessian-quadratic-form} is negative, which contradicts the second-order stationarity of the state
$\boldsymbol{\theta}$.
\end{proof}

\begin{figure}[H]
    \centering
    \includegraphics[width=0.5\linewidth]{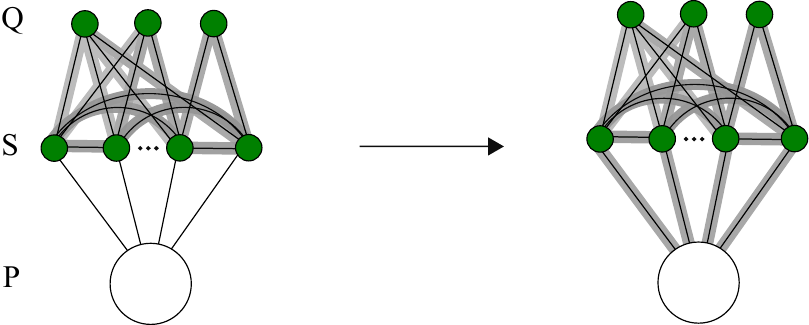}
    \caption{Illustration of Lemma~\ref{lem:sync-implies-common-alignment}. The green vertices correspond to nodes for which $\mu_i \boldsymbol{v}_i = \sum_{j \in N(i)} \boldsymbol{v}_j$ holds with $\mu_i \ge 0$. The large circular node represents the set $P$ of vertices. Gray-shaded edges between two nodes indicate that the vectors corresponding to these two nodes are synchronized. The structural relations among the sets $Q$, $S$, and $P$, together with the synchronization of $Q \sqcup S$, imply a synchronization in phase between $Q \sqcup S$ and the phasor sum over $P$.}
    \label{fig:false-twin-like-sum}
\end{figure}

\begin{lemma}\label{lemma:pendant}
Let $G=(V,E)$ be a graph, and let $W \subseteq V$ admit a partition
\[ W = Q \sqcup S_1 \sqcup S_2 \sqcup P\]
with the following structural conditions:

(i) Every node in $Q$ has all its neighbors inside $S_1$, 
    i.e., $N(i) \subseteq S_1$ for all $i \in Q$;

(ii) For each $i \in S_1 \sqcup S_2$, its neighborhood is 
    $P \sqcup N_Q(i) \sqcup (S \setminus \{i\})$, 
    where $$N_Q(i) = \{\, q \in Q : q \sim i \,\}.$$
    
(iii) For any \( i \in S_1 \sqcup S_2 \), its neighborhood can be decomposed into three parts: the set \(P\), its neighbours inside \(Q\), and the rest of the clique \(S_1 \sqcup S_2\).  
That is
\[
N(i) \;=\; P \sqcup N_Q(i) \sqcup \bigl((S_1 \sqcup S_2)\setminus \{i\}\bigr).
\]
Assume further that \( \{\boldsymbol{v}_i\}_{i\in V} \) corresponds to a second-order stationary point
\( \boldsymbol{\theta} \) of the Kuramoto energy~\eqref{E}, at which the induced subgraph on
\( Q \sqcup S_1 \) is synchronized, i.e.,
\[
\boldsymbol{v}_i = \boldsymbol{v}, \quad \forall i \in Q \sqcup S_1 ,
\]
for some common phasor \( \boldsymbol{v} \).
Then the set $S_1\sqcup S_2$ synchronize.
\end{lemma}

\begin{proof}
	\emph{Step A}. Observe that the set $S_1$ induces a clique, and $Q\sqcup S_1$ forms a synchronizing group. Hence by Lemma \ref{lem:sync-implies-common-alignment}, for every $i\in S_1$ there exists $\mu_i\ge 0$ such that
    \begin{equation}\label{parallel}
        \sum_{k\in S_2 \sqcup P} \boldsymbol v_k = \mu_i\,\boldsymbol v_i .
    \end{equation}

\emph{Step B}. Since $S_1$ is a synchronizing group, we have
\[
\sum_{k\in S_1\setminus\{i\}} \boldsymbol v_k
= (|S_1|-1)\boldsymbol v_i .
\] where $|S_1|-1\geq 0$.
Adding this term to~\eqref{parallel} yields
\[
\sum_{k\in S_2 \sqcup P \sqcup (S_1\setminus\{i\})} \boldsymbol v_k
= \bigl(\mu_i + |S_1|-1\bigr)\boldsymbol v_i .
\] where $\mu_i + |S_1|-1 > 0$.
In other words,
\begin{equation}\label{align1}
\forall i\in S_1,\qquad
\boldsymbol v_i \Uparrow
\sum_{k\in S_2 \sqcup P \sqcup (S_1\setminus\{i\})} \boldsymbol v_k .
\end{equation}

\emph{Step C}. According to Lemma \ref{stable-eq}, nodes in the set $S_2$ satisfy \begin{equation}\label{align2}
		\begin{aligned}
		\forall i\in S_2,\quad \boldsymbol{v}_i\Uparrow \sum_{k\in S_1\sqcup P \sqcup S_2 \setminus \{i\}}\boldsymbol{v}_k\quad\text{or}\quad \sum_{k\in S_1\sqcup P \sqcup S_2 \setminus \{i\}}\boldsymbol{v}_k=\boldsymbol{0}.
		\end{aligned}
	\end{equation} since $\boldsymbol{\theta}$ is a second-order stationary point.
	Combining \eqref{align1} and \eqref{align2}, we deduce that all nodes in $S_1 \sqcup S_2$ are geometrically stable twins in the sense of Definition~\ref{def:geometric_twins}. It then follows from Corollary~\ref{sync_twins} that $S_1 \sqcup S_2$ forms a synchronizing group.
\end{proof}

\begin{figure}[H]
    \centering
    \includegraphics[scale=0.38]{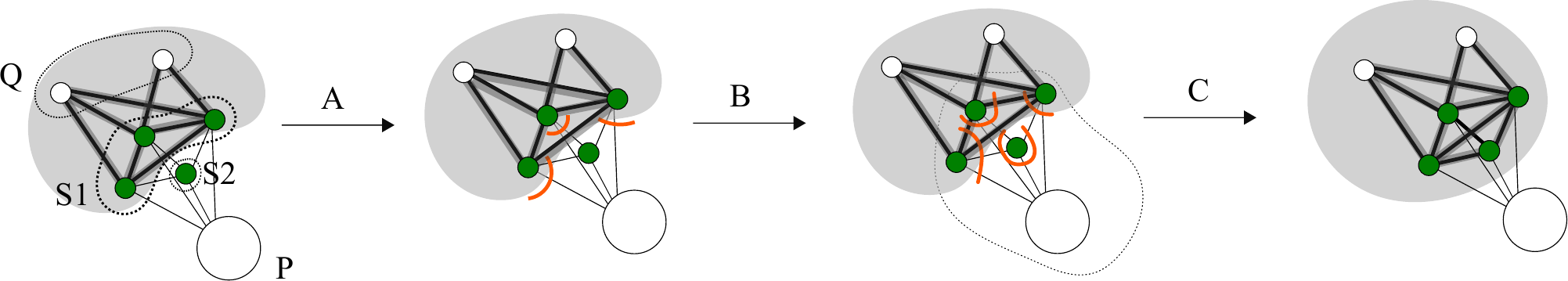}
    \label{fig:false-twin-like-sum}
    \caption{Illustration of the proof of Lemma~\ref{lemma:pendant}, illustrating how the synchronized block \( Q \sqcup S_1 \) extends to include \( S_2 \) through propagation enabled by the specific local structure. The green vertices correspond to nodes for which $\mu_i \boldsymbol{v}_i = \sum_{j \in N(i)} \boldsymbol{v}_j$ holds with $\mu_i \ge 0$. The gray-shaded nodes are synchronized; an edge with a gray background indicates synchronization between the corresponding nodes. Each orange arc separates a small green node from a group of nodes it connects to, indicating that the vector of the green node is in the same direction as the sum of the vectors of the nodes on the opposite side of the arc.}
\end{figure}
\section{Main Proof: the Induction Argument}\label{main_pf}
\subsection{A Prototype Example}
Let us take a look at the threshold graph corresponding to the sequence \texttt{01010101}. The first vertex is denoted by $A$, and we label the vertices $A$ through $I$ in the order in which they appear in the sequence.

Let $G$ be the threshold graph corresponding to the sequence \texttt{01010101}. 
To prove that $G$ is globally synchronizing, we show that every second-order stationary point $\boldsymbol{\theta}$ of $E_G(\boldsymbol{\theta})$ is synchronous; that is, 
$\theta_1 = \cdots = \theta_n$.

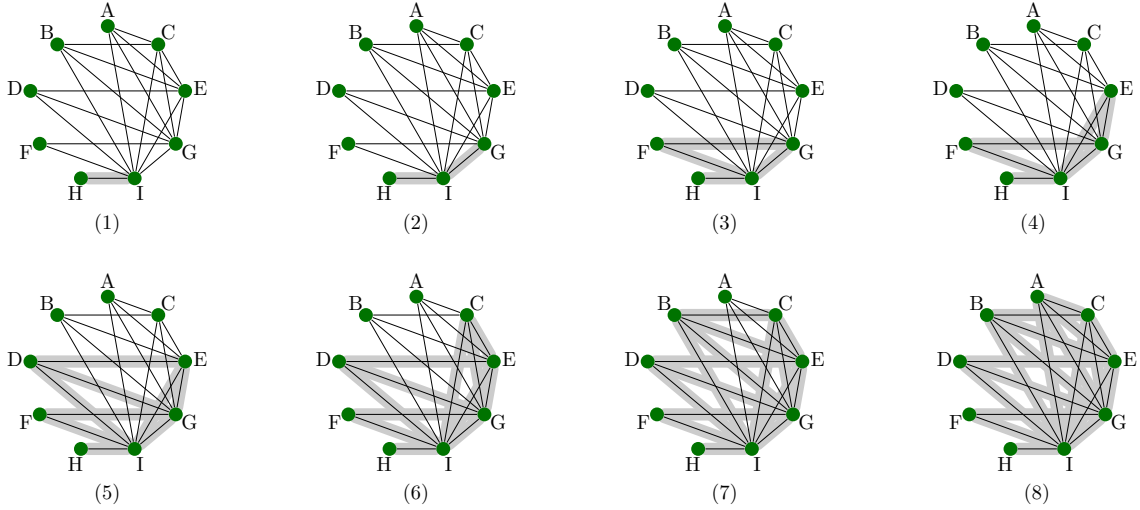
\begin{figure}[h]
\centering
\begin{adjustbox}{max width=1\textwidth}
\begin{tabular}{cccc}

\begin{subfigure}{0.3\textwidth}
\centering
\begin{tikzpicture}[scale=1.35,
    node/.style={circle, fill=green!45!black, inner sep=2.4pt},
    lbl/.style={font=\small},
    edge/.style={thin} ]

\foreach \i/\name in {0/A,1/C,2/E,3/G,4/I,5/H,6/F,7/D,8/B}{
    \pgfmathsetmacro{\ang}{90 - 40*\i}
    \node[node] (\name) at (\ang:1) {};
    \node[lbl] at (\ang:1.2) {\name};
}

\begin{pgfonlayer}{background}
    \draw[line width=6pt, gray!40, line cap=round] (I) -- (H);
\end{pgfonlayer}

\foreach \u/\v in {
A/C,A/E,A/G,A/I,
C/E,C/G,C/I,
E/G,E/I,G/I,
B/C,B/E,B/G,B/I,
D/E,D/G,D/I,
F/G,F/I,
H/I}{\draw[edge] (\u)--(\v);}

\end{tikzpicture}
\caption*{(1)}
\end{subfigure}
&
\begin{subfigure}{0.3\textwidth}
\centering
\begin{tikzpicture}[scale=1.35,
    node/.style={circle, fill=green!45!black, inner sep=2.4pt},
    lbl/.style={font=\small},
    edge/.style={thin} ]

\foreach \i/\name in {0/A,1/C,2/E,3/G,4/I,5/H,6/F,7/D,8/B}{
    \pgfmathsetmacro{\ang}{90 - 40*\i}
    \node[node] (\name) at (\ang:1) {};
    \node[lbl] at (\ang:1.2) {\name};
}

\begin{pgfonlayer}{background}
    \draw[line width=6pt, gray!40, line cap=round] (G)--(I)--(H);
\end{pgfonlayer}

\foreach \u/\v in {
A/C,A/E,A/G,A/I,
C/E,C/G,C/I,
E/G,E/I,G/I,
B/C,B/E,B/G,B/I,
D/E,D/G,D/I,
F/G,F/I,
H/I}{\draw[edge] (\u)--(\v);}

\end{tikzpicture}
\caption*{(2)}
\end{subfigure}
&
\begin{subfigure}{0.3\textwidth}
\centering
\begin{tikzpicture}[scale=1.35,
    node/.style={circle, fill=green!45!black, inner sep=2.4pt},
    lbl/.style={font=\small},
    edge/.style={thin} ]

\foreach \i/\name in {0/A,1/C,2/E,3/G,4/I,5/H,6/F,7/D,8/B}{
    \pgfmathsetmacro{\ang}{90 - 40*\i}
    \node[node] (\name) at (\ang:1) {};
    \node[lbl] at (\ang:1.2) {\name};
}

\begin{pgfonlayer}{background}
    \draw[line width=6pt, gray!40, line cap=round] (G)--(I)--(H);
    \draw[line width=6pt, gray!40, line cap=round] (G)--(F);
    \draw[line width=6pt, gray!40, line cap=round] (I)--(F);
\end{pgfonlayer}

\foreach \u/\v in {
A/C,A/E,A/G,A/I,
C/E,C/G,C/I,
E/G,E/I,G/I,
B/C,B/E,B/G,B/I,
D/E,D/G,D/I,
F/G,F/I,
H/I}{\draw[edge] (\u)--(\v);}

\end{tikzpicture}
\caption*{(3)}
\end{subfigure}
&
\begin{subfigure}{0.3\textwidth}
\centering
\begin{tikzpicture}[scale=1.35,
    node/.style={circle, fill=green!45!black, inner sep=2.4pt},
    lbl/.style={font=\small},
    edge/.style={thin} ]

\foreach \i/\name in {0/A,1/C,2/E,3/G,4/I,5/H,6/F,7/D,8/B}{
    \pgfmathsetmacro{\ang}{90 - 40*\i}
    \node[node] (\name) at (\ang:1) {};
    \node[lbl] at (\ang:1.2) {\name};
}

\begin{pgfonlayer}{background}
    \draw[line width=6pt, gray!40, line cap=round] (G)--(I)--(H);
    \draw[line width=6pt, gray!40, line cap=round] (G)--(F);
    \draw[line width=6pt, gray!40, line cap=round] (I)--(F);
    \draw[line width=6pt, gray!40, line cap=round] (E)--(G);
    \draw[line width=6pt, gray!40, line cap=round] (E)--(I);
\end{pgfonlayer}

\foreach \u/\v in {
A/C,A/E,A/G,A/I,
C/E,C/G,C/I,
E/G,E/I,G/I,
B/C,B/E,B/G,B/I,
D/E,D/G,D/I,
F/G,F/I,
H/I}{\draw[edge] (\u)--(\v);}

\end{tikzpicture}
\caption*{(4)}
\end{subfigure}
\\[1em]
\begin{subfigure}{0.3\textwidth}
\centering
\begin{tikzpicture}[scale=1.35,
    node/.style={circle, fill=green!45!black, inner sep=2.4pt},
    lbl/.style={font=\small},
    edge/.style={thin} ]

\foreach \i/\name in {0/A,1/C,2/E,3/G,4/I,5/H,6/F,7/D,8/B}{
    \pgfmathsetmacro{\ang}{90 - 40*\i}
    \node[node] (\name) at (\ang:1) {};
    \node[lbl] at (\ang:1.2) {\name};
}

\begin{pgfonlayer}{background}
    \draw[line width=6pt, gray!40, line cap=round] (G)--(I)--(H);
    \draw[line width=6pt, gray!40, line cap=round] (G)--(F);
    \draw[line width=6pt, gray!40, line cap=round] (I)--(F);
    \draw[line width=6pt, gray!40, line cap=round] (E)--(G);
    \draw[line width=6pt, gray!40, line cap=round] (E)--(I);
    \draw[line width=6pt, gray!40, line cap=round] (D)--(E);
    \draw[line width=6pt, gray!40, line cap=round] (D)--(G);
    \draw[line width=6pt, gray!40, line cap=round] (D)--(I);
\end{pgfonlayer}

\foreach \u/\v in {
A/C,A/E,A/G,A/I,
C/E,C/G,C/I,
E/G,E/I,G/I,
B/C,B/E,B/G,B/I,
D/E,D/G,D/I,
F/G,F/I,
H/I}{\draw[edge] (\u)--(\v);}

\end{tikzpicture}
\caption*{(5)}
\end{subfigure}
&
\begin{subfigure}{0.3\textwidth}
\centering
\begin{tikzpicture}[scale=1.35,
    node/.style={circle, fill=green!45!black, inner sep=2.4pt},
    lbl/.style={font=\small},
    edge/.style={thin} ]

\foreach \i/\name in {0/A,1/C,2/E,3/G,4/I,5/H,6/F,7/D,8/B}{
    \pgfmathsetmacro{\ang}{90 - 40*\i}
    \node[node] (\name) at (\ang:1) {};
    \node[lbl] at (\ang:1.2) {\name};
}

\begin{pgfonlayer}{background}
    \draw[line width=6pt, gray!40, line cap=round] (G)--(I)--(H);
    \draw[line width=6pt, gray!40, line cap=round] (G)--(F);
    \draw[line width=6pt, gray!40, line cap=round] (I)--(F);
    \draw[line width=6pt, gray!40, line cap=round] (E)--(G);
    \draw[line width=6pt, gray!40, line cap=round] (E)--(I);
    \draw[line width=6pt, gray!40, line cap=round] (D)--(E);
    \draw[line width=6pt, gray!40, line cap=round] (D)--(G);
    \draw[line width=6pt, gray!40, line cap=round] (D)--(I);
    \draw[line width=6pt, gray!40, line cap=round] (C)--(E);
    \draw[line width=6pt, gray!40, line cap=round] (C)--(G);
    \draw[line width=6pt, gray!40, line cap=round] (C)--(I);
\end{pgfonlayer}

\foreach \u/\v in {
A/C,A/E,A/G,A/I,
C/E,C/G,C/I,
E/G,E/I,G/I,
B/C,B/E,B/G,B/I,
D/E,D/G,D/I,
F/G,F/I,
H/I}{\draw[edge] (\u)--(\v);}

\end{tikzpicture}
\caption*{(6)}
\end{subfigure}
&
\begin{subfigure}{0.3\textwidth}
\centering
\begin{tikzpicture}[scale=1.35,
    node/.style={circle, fill=green!45!black, inner sep=2.4pt},
    lbl/.style={font=\small},
    edge/.style={thin} ]

\foreach \i/\name in {0/A,1/C,2/E,3/G,4/I,5/H,6/F,7/D,8/B}{
    \pgfmathsetmacro{\ang}{90 - 40*\i}
    \node[node] (\name) at (\ang:1) {};
    \node[lbl] at (\ang:1.2) {\name};
}

\begin{pgfonlayer}{background}
    \draw[line width=6pt, gray!40, line cap=round] (G)--(I)--(H);
    \draw[line width=6pt, gray!40, line cap=round] (G)--(F);
    \draw[line width=6pt, gray!40, line cap=round] (I)--(F);
    \draw[line width=6pt, gray!40, line cap=round] (E)--(G);
    \draw[line width=6pt, gray!40, line cap=round] (E)--(I);
    \draw[line width=6pt, gray!40, line cap=round] (D)--(E);
    \draw[line width=6pt, gray!40, line cap=round] (D)--(G);
    \draw[line width=6pt, gray!40, line cap=round] (D)--(I);
    \draw[line width=6pt, gray!40, line cap=round] (C)--(E);
    \draw[line width=6pt, gray!40, line cap=round] (C)--(G);
    \draw[line width=6pt, gray!40, line cap=round] (C)--(I);
    \draw[line width=6pt, gray!40, line cap=round] (B)--(C);
    \draw[line width=6pt, gray!40, line cap=round] (B)--(E);
    \draw[line width=6pt, gray!40, line cap=round] (B)--(G);
    \draw[line width=6pt, gray!40, line cap=round] (B)--(I);
\end{pgfonlayer}

\foreach \u/\v in {
A/C,A/E,A/G,A/I,
C/E,C/G,C/I,
E/G,E/I,G/I,
B/C,B/E,B/G,B/I,
D/E,D/G,D/I,
F/G,F/I,
H/I}{\draw[edge] (\u)--(\v);}

\end{tikzpicture}
\caption*{(7)}
\end{subfigure}
&
\begin{subfigure}{0.3\textwidth}
\centering
\begin{tikzpicture}[scale=1.35,
    node/.style={circle, fill=green!45!black, inner sep=2.4pt},
    lbl/.style={font=\small},
    edge/.style={thin} ]
\foreach \i/\name in {0/A,1/C,2/E,3/G,4/I,5/H,6/F,7/D,8/B}{
    \pgfmathsetmacro{\ang}{90 - 40*\i}
    \node[node] (\name) at (\ang:1) {};
    \node[lbl] at (\ang:1.2) {\name};}
\begin{pgfonlayer}{background}
    \draw[line width=6pt, gray!40, line cap=round] (G)--(I)--(H);
    \draw[line width=6pt, gray!40, line cap=round] (G)--(F);
    \draw[line width=6pt, gray!40, line cap=round] (I)--(F);
    \draw[line width=6pt, gray!40, line cap=round] (E)--(G);
    \draw[line width=6pt, gray!40, line cap=round] (E)--(I);
    \draw[line width=6pt, gray!40, line cap=round] (D)--(E);
    \draw[line width=6pt, gray!40, line cap=round] (D)--(G);
    \draw[line width=6pt, gray!40, line cap=round] (D)--(I);
    \draw[line width=6pt, gray!40, line cap=round] (C)--(E);
    \draw[line width=6pt, gray!40, line cap=round] (C)--(G);
    \draw[line width=6pt, gray!40, line cap=round] (C)--(I);
    \draw[line width=6pt, gray!40, line cap=round] (B)--(C);
    \draw[line width=6pt, gray!40, line cap=round] (B)--(E);
    \draw[line width=6pt, gray!40, line cap=round] (B)--(G);
    \draw[line width=6pt, gray!40, line cap=round] (B)--(I);
    \draw[line width=6pt, gray!40, line cap=round] (A)--(C);
    \draw[line width=6pt, gray!40, line cap=round] (A)--(E);
    \draw[line width=6pt, gray!40, line cap=round] (A)--(G);
    \draw[line width=6pt, gray!40, line cap=round] (A)--(I);
\end{pgfonlayer}
\foreach \u/\v in {
A/C,A/E,A/G,A/I,
C/E,C/G,C/I,
E/G,E/I,G/I,
B/C,B/E,B/G,B/I,
D/E,D/G,D/I,
F/G,F/I,
H/I}{\draw[edge] (\u)--(\v);}
\end{tikzpicture}
\caption*{(8)}
\end{subfigure}
\end{tabular}
\end{adjustbox}
\caption{Synchronization propagates in eight steps, forcing every second-order stationary point to be synchronous. Gray-shaded edges represent synchronizing relations that have been established at each step.}
\end{figure}

\textit{Step (1)} We begin with the edge between $H$ and $I$. 
Applying Lemma~\ref{stable-eq} to node $H$ shows that, since $\boldsymbol{\theta}$ is a second-order 
stationary point, we must have $\boldsymbol{v}_H \Uparrow \boldsymbol{v}_I$. 
Equivalently, $$\theta_H=\theta_I.$$

\textit{Step (2)} Observe that $N(G)=N(I)\cup \{H\}$ and $\theta_H=\theta_I$, applying Lemma \ref{lemma:pendant} by letting $Q=\{H\}$, $S_1=\{I\}$, $S_2=\{G\}$, $P=\{F,D,B,A,C,E\}$, we obtain $$\theta_G=\theta_I.$$ Intuitively, the only difference between the neighborhoods of $G$ and $I$  is that $I$ has one additional neighbor, namely $H$.  This extra neighbor does not obstruct the synchronization between $G$ and $I$,  because $H$ is already synchronized with $I$ and therefore behaves as a synchronous pendant at this stage.

\textit{Step (3)} Note that node $F$ connects only to $G$ and $I$, which are synchronized as established in Step~(2). 
Therefore, by Lemma~\ref{stable-eq}, we have 
$$\theta_F = \theta_G = \theta_I.$$

\textit{Step (4)}
Applying Lemma~\ref{lemma:pendant} with 
$Q=\{F,H\}$, $S_1=\{G,I\}$, $S_2=\{E\}$, and $P=\{D,B,A,C\}$, 
we conclude $$\theta_E=\theta_G=\theta_I.$$

\textit{Step (5)} Note that node $D$ connects to $E$, $G$, and $I$, which are synchronized as established in Step~(4). 
Therefore, by Lemma~\ref{stable-eq}, we have 
$$\theta_D = \theta_E = \theta_G = \theta_I.$$

\textit{Step (6)} 
Applying Lemma~\ref{lemma:pendant} with 
$Q=\{D,F,H\}$, $S_1=\{E,G,I\}$, $S_2=\{C\}$, and $P=\{A,B\}$,
we conclude $$\theta_C=\theta_E=\theta_G=\theta_I.$$

\textit{Step (7)} Note that node $B$ connects to $C$, $E$, $G$, and $I$, which are synchronized as established in Step~(6). 
Therefore, by Lemma~\ref{stable-eq}, we have 
$$\theta_B = \theta_C = \theta_E = \theta_G = \theta_I.$$

\textit{Step (8)} As in Step~(7), Lemma~\ref{stable-eq} implies that $$\theta_A=\theta_C=\theta_E=\theta_G=\theta_I$$ since these nodes form a synchronizing group and constitute all neighbors of $A$.

In conclusion, every second-order stationary point $\boldsymbol{\theta}$ is a fully synchronized state. Hence the threshold graph corresponding to \texttt{01010101} is second-order globally synchronizing. This implies that it is also globally synchronizing.

\subsection{Proof of Theorem \ref{main}: General Case}
\begin{proof}[Proof of Theorem~\ref{main}]
Any connected threshold graph can be described by a bit sequence of $1$'s and $0$'s, where the last bit must be a $1$; otherwise, the threshold graph is not connected.
For simplicity, we decompose the sequence into maximal contiguous blocks of $1$'s and $0$'s:
\[
U_1,\, I_2,\, U_2,\, \dots,\, I_{k-1},\, U_{k-1},\, I_k,\, U_k,
\]
or
\[
I_1,\, U_1,\, I_2,\, U_2,\, \dots,\, I_{k-1},\, U_{k-1},\, I_k,\, U_k,
\]
where $U_i$ denotes a subsequence of $1$'s and $I_i$ denotes a subsequence of $0$'s.
Without loss of generality, we focus on threshold graphs $G=(V,E)$ of the first sequence type; the second case follows by an entirely analogous argument.

Let $\boldsymbol{\theta}$ be a second-order stationary point of the energy function $E_G(\boldsymbol{\theta})$ of the homogeneous Kuramoto model on $G$.
We prove that $\boldsymbol{\theta}$ is a fully synchronized state by induction on $m$, where $m$ denotes the number of $1$-blocks counted from the end of the threshold sequence.

For the base case, observe that all nodes within a block of $1$'s form a set of closed twins.
Hence, at any second-order stationary point $\boldsymbol{\theta}$, the nodes in the last block $U_k$ must synchronize, as follows directly from Corollary~\ref{cor:twins-sync}.

Assume as the inductive hypothesis that the nodes in the last $m$ blocks of $1$'s, namely $U_{k-m+1}, \ldots, U_k$, are synchronized, i.e., there exists a common phasor $\boldsymbol{v}$ such that $\boldsymbol{v}_i = \boldsymbol{v}$ for all $i \in U_{k-m+1} \sqcup \cdots \sqcup U_k$. Lemma~\ref{stable-eq} then implies that the nodes in the adjacent $0$-blocks $I_{k-m+1} \sqcup \cdots \sqcup I_k$ also satisfy $\boldsymbol{v}_i = \boldsymbol{v}$.

Define
\[
\begin{aligned}
Q   &:= I_{k-m+1} \sqcup \cdots \sqcup I_k, \\
S_1 &:= U_{k-m+1} \sqcup \cdots \sqcup U_k, \\
S_2 &:= U_{k-m}, \\
P   &:= U_1 \sqcup \cdots \sqcup U_{k-m-1}
         \sqcup I_2 \sqcup \cdots \sqcup I_{k-m}.
\end{aligned}
\]
We now verify that the assumptions of Lemma~\ref{lemma:pendant} are satisfied by the above
decomposition, which follows directly from the threshold construction.

By construction, the vertices in $Q$ correspond to the most recent $0$-blocks
preceding $S_1$. Hence, every vertex in $Q$ is adjacent only to vertices in $S_1$.
Moreover, since $S_1$ and $S_2$ consist of consecutive dominating blocks,
the induced subgraph on $S_1 \cup S_2$ is a clique.
Finally, for each $i\in S_1$, its neighborhood can be decomposed as
\[
N(i) = P \;\sqcup\; \bigl((S_1\cup S_2)\setminus\{i\}\bigr) \;\sqcup\; N_Q(i),
\]
where $P$ denotes the set of vertices added earlier in the construction.
Hence, all the assumptions of Lemma~\ref{lemma:pendant} are satisfied,
and we conclude that the sets $S_1$ and $S_2$ synchronize.
In other words, the last $m+1$ blocks of $1$'s $U_{k-m}, U_{k-m+1}, \cdots, U_k$ synchronize.

Therefore, every second-order stationary point is a synchronous state, and it follows that $G$ is second-order globally synchronizing, and in particular, it is globally synchronizing.
\end{proof}

\section{Summary and Open Problems}\label{summary} In this work, we identify threshold graphs as a class of interaction networks that are globally synchronizing, despite spanning the full range of admissible edge densities from $2/n$ to $1$, and the full range of admissible minimum degree from $1$ to $n-1$.
Our analysis crucially exploits the constructive nature of threshold graphs: they form the smallest graph class closed under the successive addition of isolated and universal nodes. This closure property enables an inductive approach, whereby global synchronization is shown to be preserved when starting from a smaller synchronizing threshold graph and adding either an isolated node or a universal node. From a geometric perspective, this reveals a local-to-global propagation mechanism, whereby synchronization constraints imposed on a small subset of nodes progressively force synchrony across the entire network.

A natural question is whether this closure phenomenon extends beyond
threshold graphs. Specifically, suppose that $G$ is globally synchronizing.
Is the graph obtained from $G$ by adding a universal vertex again globally
synchronizing? We expect an affirmative answer. It seems that the added universal vertex
couples uniformly to all existing vertices, and it is therefore plausible that
any stable configuration of the enlarged graph may inherit the synchronized
structure of $G$. We leave this question for future work.

Moreover, we expect that our results may admit further generalizations. 
Possible directions include extending the analysis to oscillators evolving on higher-dimensional spheres, incorporating nonlinear interaction terms as in~\cite{Geshkovski2025} and considering broader classes of graphs. For instance, one may study systems in which oscillators are coupled through both positive and negative edges.

\bibliographystyle{alpha}
\bibliography{ref}
\addtocontents{toc}{\protect\setcounter{tocdepth}{1}}%
\setcounter{tocdepth}{1}%
\vspace*{1cm}
\end{document}